\documentclass[11pt,reqno]{amsart}

\usepackage{amsmath,amsfonts,amssymb,amscd,amsthm,amsbsy,bbm, epsf,calc,  pdfsync, enumerate}
\usepackage[pdftex]{graphicx}
\usepackage{color, float}
\usepackage{datetime}
\usepackage{dsfont} 
\usepackage{bm}
\usepackage{appendix}
\usepackage{hyperref}
\usepackage{ulem}
\usepackage{appendix}
\newtheorem{thm}{Theorem}[section]
\newtheorem{lem}[thm]{Lemma}
\newtheorem{cor}[thm]{Corollary}
\newtheorem{prop}[thm]{Proposition}

\newtheorem{ex}[thm]{Example}

\theoremstyle{definition}

\newtheorem{rmk}[thm]{Remark}

\DeclareMathOperator{\spt}{spt}
\DeclareMathOperator{\Id}{Id}

\renewcommand{\epsilon}{\varepsilon}

\newcommand{\R}{\mathbb{R}}

\newcommand{\Paren}[1]{\left(#1\right)}
\newcommand{\Curly}[1]{\left\{#1\right\}}

\newcommand{\Leb}[1]{\lvert {#1}\rvert_{\mathcal{L}}}

\newcommand{\norm}[1]{\lvert #1\rvert}
\newcommand{\TwoCoupling}[1][]{\gamma_{#1}}
\newcommand{\BarTwoCoupling}[1][]{{\bar \gamma}_{#1}}
\newcommand{\MMCoupling}[1][]{\sigma_{#1}}
\newcommand{\BarMMCoupling}[1][]{{\bar \sigma}_{#1}}
\def\Avg{A}

\newcommand{\Dominating}[1][]{\mu_{#1}}

\newcommand{\DominatingDens}[1][]{g_{#1}}
\newcommand{\DominatingCommon}{\Dominating[\wedge]}
\newcommand{\DominatingCommonDens}{\DominatingDens[\wedge]}

\newcommand{\Bary}[1][]{\nu^{#1}}
\newcommand{\BaryDens}[1][]{f^{#1}}
\newcommand{\BaryCommon}{\Bary[\wedge]}
\newcommand{\BaryCommonDens}{\BaryDens[\wedge]}
\def\BaryCommonMass{m^\wedge}
\newcommand{\BaryBar}[1][]{\overline{\nu}^{#1}}

\newcommand{\Paired}[1][]{{\mu}^{#1}_{\leq}}
\newcommand{\PairedDens}[1][]{{g}^{#1}_{\leq}}
\newcommand{\PairedCommon}{\Paired[\wedge]}
\newcommand{\PairedCommonDens}{\PairedDens[\wedge]}
\newcommand{\Coupling}[1][]{{\gamma}^{#1}}
\newcommand{\BarCoupling}[1][]{{\bar \gamma}^{#1}}
\def\CommonCoupling{{\Coupling[\wedge]}}

\newcommand{\PairedMulti}[2][]{{\mu}^{#1}_{#2\leq }}
\newcommand{\PairedMultiDens}[2][]{{g}^{#1}_{#2\leq}}

\newcommand{\WassSq}[3][2]{\mathcal{MK}^2_{#1}\Paren{#2, #3}}

\newcommand{\PartialWassSq}[3][m]{\mathcal{MK}^2_{2, #1}\Paren{#2, #3}}

\def\Func{F}
\newcommand{\Functional}[2][]{\Func_{#1}(#2)}

\newcommand{\Measures}[1][m]{\mathcal{P}^{#1}}
\newcommand{\Constraint}[1][m]{\mathcal{P}_{ac}^{#1}}

\DeclareMathOperator{\optimal}{Opt}
\newcommand{\Opt}[3][]{\optimal_{#1}\Paren{#2, #3}}
\newcommand{\OptMulti}[2][]{\optimal_{#1}\Paren{#2}}
\newcommand{\Cost}[2][]{\mathcal{C}_{#1}{\Paren{#2}}}
\def\cost{c}
\newcommand{\Couples}[2]{\Pi\Paren{#1, #2}}
\newcommand{\PartialCouples}[2]{\Pi_{\leq}\Paren{#1, #2}}
\newcommand{\MultiCouples}[1]{\Pi\Paren{#1}}
\newcommand{\PartialMultiCouples}[1]{\Pi_{\leq}\Paren{#1}}
\newcommand{\Proj}[1]{\pi_{#1}}

\newcommand{\DiagMap}{\Paren{\Id\times \Id}}

\newcommand{\Mapping}[1][]{T^{#1}}
\newcommand{\MultiMapping}[2][]{T^{#1}_{#2}}
\newcommand{\MappingInv}[1][]{\Paren{T^{#1}}^{-1}}

\newcommand{\DualMapping}[1][]{S^{#1}}
\newcommand{\MultiDualMapping}[2][]{S^{#1}_{#2}}

\newcommand{\MultiDualMappingInv}[2][]{\Paren{S^{#1}_{#2}}^{-1}}
\newcommand{\MappingCoupling}[1][]{\Paren{\Id\times \Mapping[#1]}}
\newcommand{\RevMappingCoupling}[1][]{\Paren{\DualMapping[#1]\times\Id}}
\newcommand{\MultiMappingCoupling}[2][]{\Paren{\Id\times \MultiMapping[#1]{#2}}}
\newcommand{\RevMultiMappingCoupling}[2][]{\Paren{\MultiDualMapping[#1]{#2}\times\Id}}

\def\StaySet{I}

\newcommand{\mass}[1]{\mathcal{M}\Paren{#1}}
\def\m{m}
\def\mbar{\bar m}
\def\MinSmaller{\mathcal{A}_{\{\DominatingCommonDens<\BaryDens[i]\}}}
\def\MinBigger{\mathcal{A}_{\{\DominatingCommonDens>\BaryDens[i]\}}}

\def\Number{N}
\def\Indices{I}

\def\Indicator{\mathbbm{1}}
\begin{document}
\author{Jun Kitagawa}
\address{Department of Mathematics, University of Toronto,
Toronto, Ontario, Canada M5S 2E4}
\email{kitagawa@math.toronto.edu}

\author{Brendan Pass}
\address{Department of Mathematical and Statistical Sciences, 632 CAB, University of Alberta, Edmonton, Alberta, Canada, T6G 2G1}
\email{pass@ualberta.ca}
\title{The multi-marginal optimal partial transport problem}
\thanks{This material is based upon work supported by the National Science Foundation under Grant No. 0932078 000, while both authors were in residence at the Mathematical Sciences Research Institute in Berkeley, California, during
the Fall 2013  program on Optimal Transport: Geometry and Dynamics.  We would like to thank both MSRI and the organizers of the program for their generous  hospitality. In addition, B.P. is pleased to acknowledge the support of a University of Alberta start-up grant and National Sciences and Engineering Research Council of Canada Discovery Grant number 412779-2012.}

\begin{abstract}
We introduce and study a multi-marginal optimal partial transport problem.  Under a natural and sharp condition on the dominating marginals, we establish uniqueness of the optimal plan.   Our strategy of proof establishes and exploits a connection with another novel problem, which we call the Monge-Kantorovich partial barycenter problem (with quadratic cost).  This latter problem has a natural interpretation as a variant of the mines and factories description of optimal transport.  We then turn our attention to various analytic properties of these two problems.  Of  particular interest, we show that monotonicity of the active marginals with respect to the amount $m$ of mass to be transported can fail, a surprising difference from the two marginal case.
\end{abstract}
\maketitle
\section{Introduction}
Throughout this paper, whenever we write ``measure'' it will tacitly be assumed that we are referring to a positive, Borel measure on the relevant space in question. In all but the last section, we also assume a measure $\mu$, when it is defined on $\R^n$, has finite second moment; i.e. that $\int_{\R^n}\norm{x}^2\mu(dx)<\infty$. Also, ``absolutely continuous'', ``a.e.'', ``null set'', and ``zero measure'' without any further qualifiers will always be with respect to the Lebesgue measure. Finally, for any measure $\mu$, we will write
\begin{align*}
\mass{\mu}:=\mu(X),
\end{align*}
where $X$ is the entire space that $\mu$ is defined on.

Recall the classical optimal transport problem with quadratic cost: let $\Dominating$ and $\Bary$ be measures on $\R^n$ satisfying the mass constraint $\mass{\Dominating}=\mass{\Bary}<\infty$, and write $\Couples{\Dominating}{\Bary}$ for the collection of all measures on $\R^n\times\R^n$ whose left and right marginals equal $\Dominating$ and $\Bary$, respectively. Then a solution of the optimal transport problem (with quadratic cost) is a measure $\TwoCoupling\in\Couples{\Dominating}{\Bary}$ achieving the minimum value in
\begin{align}\label{eqn: OT}\tag{OT}
\WassSq{\Dominating}{\Bary}:=\min_{\TwoCoupling'\in\Couples{\Dominating}{\Bary}}\int_{\R^n\times\R^n} \norm{x-y}^2\TwoCoupling'(dx, dy).
\end{align}
We will denote the collection of solutions to~\eqref{eqn: OT} above as $\Opt{\Dominating}{\Bary}$.  Existence of an optimizer is not difficult to show; a famous theorem of Brenier  and McCann implies that if the first measure is absolutely continuous, and both measures have finite second moments, the solution is unique and is in fact concentrated on the graph $\{(x,T(x))\}$ of a function over the first variable \cite{Bre87,Bre91}\cite{McC95}.  This result has been extended to a wide class of other costs (see, for example, \cite{Gan95, Gm96, Caf96}).

We will be concerned here with two natural extensions of~\eqref{eqn: OT} above. The first is the  optimal \emph{partial} transport problem: let $\Dominating$ and $\Bary$ be measures on $\R^n$ each with finite total mass (not necessarily equal), fix any $0\leq m\leq \min{\{\mass{\Dominating},\ \mass{\Bary}\}}$, and write $\PartialCouples{\Dominating}{\Bary}$ for the collection of all measures on $\R^n\times\R^n$ whose left and right marginals are \emph{dominated} by $\Dominating$ and $\Bary$ respectively, that is, $(\Proj{1})_\#\TwoCoupling\Paren{E}\leq \Dominating\Paren{E}$ and $(\Proj{2})_\#\TwoCoupling\Paren{E}\leq \Bary\Paren{E}$ for any measurable set $E$, where $\Proj{j}$ denotes projection onto the $j$th coordinate (for ease of notation we will simply write $\Paired\leq\Dominating$ to indicate that $\Paired$ is dominated by $\Dominating$).
Then a solution of the optimal partial transport problem (again with quadratic cost) is a measure $\TwoCoupling\in\PartialCouples{\Dominating}{\Bary}$ with $\mass{\TwoCoupling}=m$ achieving the minimum value in
\begin{align}\label{eqn: OTm}\tag{$\rm{OT_m}$}
\PartialWassSq{\Dominating}{\Bary}:=\min_{\TwoCoupling'\in\PartialCouples{\Dominating}{\Bary},\  \mass{\TwoCoupling'}=m}\int_{\R^n\times\R^n} \norm{x-y}^2\TwoCoupling'(dx, dy).
\end{align}
We will denote the collection of solutions to~\eqref{eqn: OTm} above as $\Opt[m]{\Dominating}{\Bary}$.  Again, existence of an optimal measure can be established in a straightforward way.  Uniqueness is much more involved; however, when the supports of the two measures are separated by a hyperplane, Caffarelli and McCann established a uniqueness result (in addition to several properties of the minimizer, see~\cite{CM10}). This assumption on the measures was weakened by Figalli in~\cite{Fig10a}; he assumed only that the pointwise minimum of the two measures has total mass not greater than  $m$.  This is easily seen to be a sharp condition for uniqueness;  if it fails, then for any measure $\Paired$ with mass $m$ satisfying both $\Paired \leq \Dominating$ and $\Paired\leq\Bary$, the diagonal coupling $\DiagMap_\#\Paired$ is clearly optimal. In addition, Figalli extended his results to a larger class of cost functions.

On the other hand, one can also consider the \emph{multi-marginal} optimal transport problem: let $\Dominating[j]$ for $j=1,\ldots, \Number$ be measures on $\R^n$ all with equal, finite mass, and write $\MultiCouples{\Dominating[1],\ldots,\Dominating[\Number]}$ for the collection of all measures on $\Paren{\R^n}^\Number$ whose $j$th marginal equals $\Dominating[j]$. Then a solution of the multi-marginal optimal transport problem, with 
cost:
\begin{align*}
\cost(x_1,\ldots, x_\Number):=\sum_{j\neq k}^\Number{\norm{x_j-x_k}^2},
\end{align*}
 (see~\cite{GS98}), is a measure $\MMCoupling\in\MultiCouples{\Dominating[1],\ldots,\Dominating[\Number]}$ achieving the minimum value in
\begin{align}\label{eqn: MM}\tag{MM}
\min_{\MMCoupling'\in\MultiCouples{\Dominating[1],\ldots,\Dominating[\Number]}}\Cost{\MMCoupling'},
\end{align}
where
\begin{align*}
\Cost{\MMCoupling'}:=\int_{\Paren{\R^n}^\Number} \cost(x_1,\ldots, x_\Number)\MMCoupling'(dx_1,\ldots, dx_\Number).
\end{align*}
Once more, existence can be established in a straightforward way.  Assuming the first measure is absolutely continuous, Gangbo and {\'S}wi{\c{e}}ch proved that the optimizer is unique  and, like in the two marginal case, is concentrated on a graph $\{(x_1,T_2(x_1),\ldots,T_\Number(x_1))\}$ over the first variable \cite{GS98},  extending earlier partial results in the same setting in \cite{OR93}\cite{KS94}.   This result has been extended to certain other cost functions \cite{C03, P11, P13a, KP13}; these costs are very special, however, and for a variety of other costs, counterexamples to uniqueness and the graphical structure are known \cite{CN08, P13b, P12}, indicating that these properties depend delicately on the cost function for multi-marginal problems.

In this paper, we combine these two extensions~\eqref{eqn: OTm} and~\eqref{eqn: MM} and consider the \emph{multi-marginal optimal partial transport problem}: let $\Dominating[j]$ for $j=1,\ldots, \Number$ be measures on $\R^n$ all with finite (but not necessarily equal) total mass, fix any $0\leq m\leq \min_{1\leq j\leq \Number}{\{\mass{\Dominating[j]}\}}$, and write $\PartialMultiCouples{\Dominating[1],\ldots,\Dominating[\Number]}$ for the collection of all measures on $\Paren{\R^n}^\Number$ whose $j$th marginal is dominated by $\Dominating[j]$ for each $j$. Then a solution of the multi-marginal optimal partial transport problem  can be defined as a measure $\MMCoupling\in\PartialMultiCouples{\Dominating[1],\ldots,\Dominating[\Number]}$ with $\mass{\MMCoupling}=m$ achieving the minimum value in
\begin{align}\label{eqn: MMm}\tag{$\rm{MM_m}$}
\min_{\MMCoupling'\in\PartialMultiCouples{\Dominating[1],\ldots,\Dominating[\Number]},\ \mass{\MMCoupling'}=m}\Cost{\MMCoupling'}.
\end{align}
Analogously to the above, we will denote the collection of solutions to~\eqref{eqn: MMm}  as $\OptMulti[m]{\Dominating[1],\ldots,\Dominating[\Number]}$. In informal exposition, we will sometimes refer to any marginal of a minimizer in either~\eqref{eqn: OTm} or~\eqref{eqn: MMm} as an ``active submeasure.''

As in~\eqref{eqn: OTm}, existence of a minimizer in~\eqref{eqn: MMm} is not difficult to see; the first issue one encounters is that of uniqueness, which will be the focus of this paper. Our main goal is to identify conditions under which the multi-marginal problem~\eqref{eqn: MMm} admits a unique solution; it turns out that a condition analogous to the one given by Figalli in~\cite{Fig10a} is sufficient, see Theorem~\ref{thm: partial barycenter is unique} below.

Our approach here  involves the analysis of another problem, which turns out to be essentially equivalent to~\eqref{eqn: MMm} and which we call the \emph{(Monge-Kantorovich) partial barycenter problem}.  This is a natural extension of the usual (Monge-Kantorovich) barycenter problem, which is, given measures $\Dominating[1],\ldots, \Dominating[\Number]$, all with mass $\m$, to find a minimizer of
\begin{align}\label{eqn: BC}\tag{BC}
\min_{\Bary\in\Measures}\sum_{j=1}^\Number\WassSq{\Dominating[j]}{\Bary},
\end{align}
where 
\begin{align*}
 \Measures:=\{\text{all measures }\Bary\mid \mass{\Bary}=\m,\ \int_{\R^n}\norm{x}^2\Bary(dx)<\infty\}.
\end{align*}
This problem was introduced by Agueh and Carlier, who showed it is essentially equivalent to \eqref{eqn: MM} (see~\cite{AC11}). 

We introduce the appropriate analogue, the \emph{partial barycenter}, as a minimizer in
\begin{align}\label{eqn: BCm}\tag{$\rm{BC_m}$}
\min_{\Bary\in\Measures}\Functional[\m]{\Bary, \Dominating[1],\ldots, \Dominating[\Number]},
\end{align}
where $m$ and the $\Dominating[j]$ are as in \eqref{eqn: MMm} and,
\begin{align*}
\Functional[\m]{\Bary, \Dominating[1],\ldots, \Dominating[\Number]}:=\sum_{j=1}^\Number\PartialWassSq{\Dominating[j]}{\Bary}.
\end{align*}
When the collection of measures $\Dominating[1],\ldots, \Dominating[\Number]$ and the mass constraint $\m$ is clear, we will suppress them and simply write $\Functional{\Bary}$ in place of $\Functional[\m]{\Bary, \Dominating[1],\ldots, \Dominating[\Number]}$. Also, we may sometimes refer to the submeasures of $\Dominating[j]$ that are actually coupled to a minimizer of~\eqref{eqn: BCm} as ``active submeasures'' as well.

We will first show there is a connection between the problems~\eqref{eqn: MMm} and~\eqref{eqn: BCm} (which is analogous to the relationship between~\eqref{eqn: MM} and~\eqref{eqn: BC} in \cite{AC11}), expressed by the following theorem:

\begin{prop}[Equivalence of {\eqref{eqn: BCm}} and {\eqref{eqn: MMm}}]\label{prop: BCm=MMm}
 For $1\leq j\leq \Number$, fix absolutely continuous measures $\Dominating[j]$, and some $0<\m\leq \min_{1\leq j\leq \Number}{\mass{\Dominating[j]}}$.

Then for any optimal measure $\sigma$ in~\eqref{eqn: MMm}, $\Avg_{\#}\sigma$ is optimal in~\eqref{eqn: BCm}, where 
\begin{align}\label{eqn: average}
\Avg(x_1,\ldots, x_\Number) := \frac{1}{\Number}\sum_{j=1}^\Number x_j.
\end{align}
  Conversely, for any minimizer $\Bary$ in~\eqref{eqn: BCm}, the measure $\Paren{\MultiDualMapping[\Bary]{1},\ldots, \MultiDualMapping[\Bary]{\Number}}_{\#}\Bary$ is optimal in~\eqref{eqn: MMm}, where $\MultiDualMapping[\Bary]{j}$ is the optimal mapping such that $\RevMultiMappingCoupling[\Bary]{j}_\#\Bary\in\Opt[m]{\Dominating[j]}{\Bary}$ for each $1\leq j\leq \Number$.
 
 Furthermore, the minimizer of \eqref{eqn: MMm} is unique if and only the minimizer of \eqref{eqn: BCm} is unique.
\end{prop}

Then, we will turn to the question of uniqueness in~\eqref{eqn: BCm}. We establish the following theorem, which shows that under conditions analogous to those in~\cite{Fig10a}, we indeed obtain uniqueness in~\eqref{eqn: BCm}:
\begin{thm}[Uniqueness of partial barycenters]\label{thm: partial barycenter is unique}
 For $1\leq j\leq \Number$, fix absolutely continuous measures $\Dominating[j]$, each with finite mass and with densities $\DominatingDens[j]$. Writing $\DominatingCommon$ for the absolutely continuous measure with density $\DominatingCommonDens:=\min_{1\leq j\leq \Number}{\DominatingDens[j]}$, fix some $\m\geq0$ satisfying 
 \begin{align*}
 \mass{\DominatingCommon}\leq\m\leq \min_{1\leq j\leq \Number}{\{\mass{\Dominating[j]}\}}.
 \end{align*}
  Then there exists a unique minimizer in $\Measures$ of~\eqref{eqn: BCm}.
\end{thm}
Finally, by combining Proposition~\ref{prop: BCm=MMm} with Theorem~\ref{thm: partial barycenter is unique}, we immediately obtain the following corollary:
\begin{cor}\label{cor: uniqueness}
Under the assumptions of Theorem~\ref{thm: partial barycenter is unique}, the multi-marginal optimal partial transport problem~\eqref{eqn: MMm} has a unique solution.
\end{cor}
Surprisingly, certain ``monotonicity properties '' enjoyed by solutions of~\eqref{eqn: OTm} are not exhibited by solutions of~\eqref{eqn: MMm}; we will briefly demonstrate this fact with some examples later.

 One might expect that an alternature approach, following the work of Figalli \cite{Fig10a}, could be used to establish Corollary \ref{cor: uniqueness}; more precisely, that one could show that the function $ m \mapsto \min_{\MMCoupling'\in\PartialMultiCouples{\Dominating[1],\ldots,\Dominating[\Number]},\ \mass{\MMCoupling'}=m}\Cost{\MMCoupling'}$ is strictly convex on  $[\mass{\DominatingCommon}, \min_{1\leq j\leq \Number}{\{\mass{\Dominating[j]}\}}]$ and use this fact to deduce uniqueness of the optimal plan in \eqref{eqn: MMm}.  As one of our examples illustrates (see Remark \ref{classcouplingnotoptimal} below), it turns out that the natural multi-marginal analogue of a key preliminary result of Figalli (Proposition 2.4 in \cite{Fig10a}) \textit{does not} hold.  As this proposition is used in a crucial way in the proof of Figalli's main result (\cite[Theorem 2.10]{Fig10a}), a direct extension of his techniques cannot be used to prove Corollary \ref{cor: uniqueness}.

We pause now to describe an economic interpretation of the partial barycenter problem, in the context of the well known factories-and-mines interpretation of the classical optimal transport problem.
\subsection{Interpretation of the partial barycenter problem}  
The optimal transport problem is frequently interpreted as the problem of matching the production of a resource (say iron ore) by a distribution of mines over a landscape $M \subset \mathbb{R}^n$ with consumption of that resource by a distribution of factories over the same landscape.   These distributions are represented by measures $\Dominating$ and $\Bary$, respectively.  The cost function $c(x,y)$ ($=\norm{x-y}^2$ in our setting) represents the cost to move one unit of iron from a mine at position $x$ to a factory at position $y$.   If the total production capacity of the mines matches the total consumption capacity by the factories (that is, the total masses of $\Dominating$ and $\Bary$ coincide), and one would like to use all of the produced resources, the problem of determining which mine should supply which factory to minimize the total transportation cost is represented by ~\eqref{eqn: OT}. More realistically, the total production capacity of the mines may not match the total consumption capacity of the factories, and one may only wish to consume a smaller portion $\m$ of the total capacity; in this case, the analogous problem is represented by ~\eqref{eqn: OTm}, as is discussed in \cite{CM10}.

Suppose now that production of a certain good requires several resources; for example, iron, aluminum, and nickel, and that the company has not yet built their factories (and so is free to build them at any locations they choose).  Production capacity of the resources are given by distributions $\Dominating[j]$ of mines over a landscape $M \subseteq \mathbb{R}^n$, for $j=1,2,3,\ldots,\Number$.  Given costs $c_j(x_j,y)$ ($\norm{x_j-y}^2$ here, but see also the extension to more general costs in Section~\ref{sec: general costs}) to move a unit of resource $j$ from a mine at position $x_j$ to a (potential) factory at position $y$, the company now wishes to \textit{build} a distribution of factories, $\Bary$, where these resources will be consumed, in order to minimize the sum of all the total transportation costs; if the total production of each resource is the same and all produced resources are to be consumed, this amounts to the barycenter problem ~\eqref{eqn: BC}. 

 However, if, perhaps because of limited demand for the good in question, only a fixed portion $\m$ of each resource is to be consumed (less than the smallest total production capacity of the resources, which may now differ for different $j$), one obtains the partial barycenter problem ~\eqref{eqn: BCm}.
\subsection{Organization of the paper}

The remainder of this paper is organized as follows. In Section~\ref{sec: connection} we will establish Proposition~\ref{prop: BCm=MMm}.  Section~\ref{sec: uniqueness in BCm} is then devoted to the proof of Theorem~\ref{thm: partial barycenter is unique}. In Section~\ref{sec: counterexamples}, we discuss some other properties of interest of minimizers of~\eqref{eqn: MMm} and ~\eqref{eqn: BCm}.  We first present two counterexamples to a ``monotonicity property'' of active submeasures (see Proposition~\ref{prop: nonmonotone} for a precise statement),
%
followed by a discussion of points where the active submeasures fail to saturate the prescribed measures $\Dominating[j]$.  These examples are in stark constrast to the two marginal partial transport problem \eqref{eqn: OTm}. We close the section with a brief remark on regularity properties of the ``free boundary''. Finally, in Section~\ref{sec: general costs} we discuss an extension of our main results to more general cost functions.

\section{Connection between multi-marginal optimal partial transport and the partial barycenter}\label{sec: connection}
For technical reasons, we will find it more convenient to work with absolutely continuous measures, hence we define the following notation:
\begin{align*}
 \Constraint:=\{\text{absolutely continuous }\Bary\mid \mass{\Bary}=\m,\ \int_{\R^n}\norm{x}^2\Bary(dx)<\infty\}.
\end{align*}
A simple argument now shows that any minimizer in~\eqref{eqn: BCm} is actually absolutely continuous, hence it is equivalent to make the minimization over $\Constraint$ in the problem, rather than $\Measures$.
\begin{lem}\label{lem: absolute continuity of BCm minimizers}
If all $\Dominating[j]$ for $1\leq j\leq \Number$ are absolutely continuous, then  any $\Bary \in \Measures$ which minimizes $\Bary \mapsto \Functional[\m]{\Bary, \Dominating[1],\ldots, \Dominating[\Number]}$ is absolutely continuous.
\end{lem}
\begin{proof} 

Note that any such $\Bary$ is necessarily optimal in the classical barycenter problem~\eqref{eqn: BC} for the active submeasures.  As these are necessarily absolutely continuous by the absolute continuity of the $\Dominating[j]$, the result in  \cite[Theorem 3.3]{P13a} (see also \cite[Theorem 5.1]{AC11} ) implies the absolute continuity of $\Bary$.
\end{proof}
With the above result in hand, we can now show Proposition~\ref{prop: BCm=MMm}.
\begin{proof}[Proof of Proposition~\ref{prop: BCm=MMm}]
Fix $\m$ as in the statement of the proposition. It is straightforward to verify that, for any $x_1,\ldots,x_\Number \in \mathbb{R}^n$,
\begin{equation}
\sum_{j\neq k}\norm{x_j -x_k}^2 = 2N\min_{y \in \mathbb{R}^n}\sum_{j=1}^\Number\norm{x_j-y}^2
\end{equation}
and that the minimum on the right hand side is attained uniquely at $y=\Avg(x_1,\ldots, x_\Number)$ (recall the definition~\eqref{eqn: average}).    The proof of the first two assertions is then a straightforward adapatation of the argument of Carlier and Ekeland in~\cite[Proposition 3]{CE10} (in fact, the three bullet points below closely mirror the three assertions in the proof of~\cite[Proposition 3]{CE10}), and we only list the main steps:
\begin{itemize}
\item For any $\Bary\in\Constraint$, let $\MultiDualMapping[\Bary]{j}$ be the optimal mapping satisfying $\RevMultiMappingCoupling[\Bary]{j}_\#\Bary\in\Opt[m]{\Dominating[j]}{\Bary}$. The measure $\MMCoupling :=\Paren{\MultiDualMapping[\Bary]{1},\ldots, \MultiDualMapping[\Bary]{\Number}}_{\#}\Bary$ is admissible in~\eqref{eqn: MMm}, and $\Cost{\MMCoupling} \leq \Functional{\Bary}$, hence (also recalling Lemma~\ref{lem: absolute continuity of BCm minimizers}) the minimum value in~\eqref{eqn: MMm} is less than the minimum value in~\eqref{eqn: BCm}.  
\item For any minimizing $\MMCoupling$ in~\eqref{eqn: MMm}, if we define $\Bary := \Avg_{\#}\MMCoupling$, we then have $\Cost{\MMCoupling} = \Functional{\Bary}$, and in light of the above, the minimum values in~\eqref{eqn: MMm} and~\eqref{eqn: BCm} are equal, and $\Bary$ is a minimizer in~\eqref{eqn: BCm}.  It also follows that $\Paren{\Proj{j}\times\Avg}_{\#}\MMCoupling$ is a minimizer in the partial transport problem~\eqref{eqn: OTm} between $\Dominating[j]$ and $\Bary$, where we move mass $m = \mass{\Bary}$; i.e. it belongs to $\Opt[\m]{\Dominating[j]}{\Bary}$.
\item For any minimizing $\Bary\in\Constraint$ in~\eqref{eqn: BCm}, the measure $\MMCoupling :=\Paren{\MultiDualMapping[\Bary]{1},\ldots, \MultiDualMapping[\Bary]{\Number}}_{\#}\Bary$ must now be minimizing in~\eqref{eqn: MMm} by the above two points.
\end{itemize}

Turning to the uniqueness assertion, we first assume that the solution $\Bary$ to \eqref{eqn: BCm} is unique.  Note that solutions to~\eqref{eqn: MMm} are in particular optimal in the regular multi-marginal problem \eqref{eqn: MM} for their marginals.  Uniqueness of minimizers in~\eqref{eqn: MM} (see \cite{GS98}) then implies that, if $\MMCoupling$ and $\bar \MMCoupling$ are distinct minimizers in~\eqref{eqn: MMm}, at least one of their marginals must differ.  

Since $\Avg_{\#}\MMCoupling$ and $\Avg_{\#}\bar \MMCoupling$ are minimizers in~\eqref{eqn: BCm} by the above, by our uniqueness assumption we have $\Avg_{\#}\MMCoupling =\Avg_{\#}\bar \MMCoupling=\Bary $. Additionally, $\TwoCoupling[j] := \Paren{\Proj{j}\times\Avg}_{\#}\MMCoupling$ and $\BarTwoCoupling[j] := \Paren{\Proj{j}\times\Avg}_{\#}\bar \MMCoupling$ both belong to $\Opt[\m]{\Dominating[j]}{\Bary}$.  Since clearly $\m$ satisfies the hypothesis of the uniqueness result~\cite[Proposition 2.2 and Theorem 2.10]{Fig10a}, we then have $\TwoCoupling[j] = \BarTwoCoupling[j]$ for all $j$, and in particular, for all $j$ the marginals $\Paren{\Proj{j}}_{\#}\MMCoupling$ and $\Paren{\Proj{j}}_{\#}\BarMMCoupling$  must coincide, which is a contradiction.

Conversely, suppose the solution $\MMCoupling$ to \eqref{eqn: MMm} is unique.  If $\Bary$ and $\bar \Bary\in\Constraint$ both minimize \eqref{eqn: BCm}, we have that $\Paren{\MultiDualMapping[\Bary]{1}, \ldots, \MultiDualMapping[\Bary]{\Number}}_\#\Bary =\MMCoupling=\Paren{\MultiDualMapping[\bar\Bary]{1}, \ldots, \MultiDualMapping[\bar\Bary]{\Number}}_\#\bar \Bary $.  In particular, note that $\Paren{\MultiDualMapping[\Bary]{j}}_ \# \Bary =\Paren{\Proj{j}}_\#\MMCoupling =  \Paren{\MultiDualMapping[\bar\Bary]{j}}_ \# \bar \Bary$ and both $\Bary$ and $\bar \Bary$ solve the regular barycenter problem \eqref{eqn: BC}, for the measures $\Paren{\MultiDualMapping[\Bary]{j}}_\# \Bary= \Paren{\MultiDualMapping[\bar\Bary]{j}}_\# \bar \Bary$ in place of the $\Dominating[j]$. Thus it follows from the uniqueness result~\cite[Proposition 3.5]{AC11} that $\Bary =\bar \Bary$.
\end{proof}

\section{Uniqueness of the partial barycenter}\label{sec: uniqueness in BCm}
We now turn to the proof of Theorem~\ref{thm: partial barycenter is unique}, uniqueness of the partial barycenter in the minimization problem~\eqref{eqn: BCm}. Throughout this section, for any measure $\Coupling$ on $\R^n\times\R^n$, we will use the following notation:
\begin{align*}
 \Cost[2]{\Coupling}:=\int_{\R^n\times\R^n} \norm{x-y}^2\Coupling(dx, dy).
\end{align*}

\subsection{Outline of the uniqueness proof}
As our proof is fairly long and technical, we believe it may be useful to the reader to provide an informal outline of the main ideas:
\begin{enumerate}[(I)]
\item We first show that the function  
\begin{equation}\label{eqn: partial transport function}
\Bary \mapsto \PartialWassSq{\Dominating}{\Bary}
\end{equation}
 is convex on $\Constraint$  (along linear interpolations); see Lemma \ref{lem: convexity claims}.
\item Furthermore,  \emph{strict} convexity of \eqref{eqn: partial transport function} can fail along an interpolant $\Bary[t]=(1-t)\Bary[0]+t\Bary[1]$ only under very special circumstances. Roughly speaking, if strict convexity fails along $\Bary[t]$, the optimal partial transport maps $\DualMapping[i]$ from $\Bary[i]$ to $\mu$, for $i=0$ and $1$ must \emph{coincide} on the intersection of the supports of $\Bary[0]$ and $\Bary[1]$. Additionally, the map $\DualMapping[0]$ is the \emph{identity} mapping wherever the density of $\Bary[0]$ strictly dominates that of $\Bary[1]$ (with a symmetric conclusion when the roles of $\Bary[0]$ and $\Bary[1]$ are reversed). This is the second part of Lemma \ref{lem: convexity claims}.
\item  Now the function $F$ to be minimized is a sum of terms of the form \eqref{eqn: partial transport function}. Thus, the convexity shown in (I) implies that if there are two distinct minimizers $\Bary[0]$ and $\Bary[1]$ of $F$,  each summand must actually be \emph{linear} along the (linear) interpolant between the two minimizers.
%
Therefore, the special properties in (II) above hold for the two marginal optimal partial transport maps from $\Bary[i]$ to \emph{each} marginal $\Dominating[j]$.   In turn, this implies that on the set where the densities $\BaryDens[0]$ of $\Bary[0]$ and $\BaryDens[1]$ of $\Bary[1]$ are \emph{not equal},  both densities  are dominated by each $\DominatingDens[j]$ (the density of $\Dominating[j]$) and thus by $\DominatingCommonDens$ (the density of $\DominatingCommon$). 
%
This is the assertion \eqref{eqn: barycenter dominated by pointwise minimum on stay set} in the body of the proof.

\item Now, as the measures $\Bary[0]$ and $\Bary[1]$ are \emph{different}, we conclude that for each of $i=0$ and $1$ there must be a set of positive Lebesgue measure where $\BaryDens[i] < \DominatingCommonDens$. 

(If this failed, say for $i=0$, (III) would imply $\BaryDens[0] = \DominatingCommonDens$ a.e. on the set where  $\BaryDens[0] \neq \BaryDens[1]$.  But then, as $\BaryDens[1] \leq \DominatingCommonDens =\BaryDens[0]$  on this set (again by (III)),  the fact that $\mass{\Bary[0]} =\mass{\Bary[1]} $,  would  imply $\BaryDens[0] = \BaryDens[1]$ a.e. on this set.  This would imply that $\Bary[0]=\Bary[1]$, a contradiction.)

\item  Now if $m> \mass{\DominatingCommon}$, for one of $i=0$ or $1$ we can find two disjoint sets with special properties.  The first is a set $\MinBigger$ on which $\DominatingCommonDens>\BaryDens[i]$ and all associated two marginal optimal partial transport maps are the identity mapping; the existence of this first set essentially follows by combining points (IV) and (II) above.  The second is a set $\MinSmaller$ on which $\DominatingCommonDens<\BaryDens[i]$ .  We also demand a certain ``mass balance'' property \eqref{eqn: bigger smaller sets} between the sets and a specific property \eqref{eqn: doesn't get mass from elsewhere} of the image of the first set under the optimal partial transport maps  -- essentially, that the optimal partial transport map from $\Bary[i]$ to each $\Dominating[j]$ doesn't move any mass from outside of $\MinBigger$ into it.  (Claim 1 in the proof; this strongly relies on \eqref{eqn: barycenter dominated by pointwise minimum on stay set}).

\item  We now modify the minimizer $\Bary[i]$ by ``filling up'' the part of the measure supported on $\MinBigger$ to match $\DominatingCommon$, and ``sucking away'' a part of the measure supported on $\MinSmaller$ of the same mass (this is the measure $\BaryBar$); pushing forward this measure by the original optimal partial transport maps gives rise to an admissible competitor $\BaryBar \in \Measures$ in the minimization of $F$ (Claim 2 in the proof, which crucially uses property \eqref{eqn: doesn't get mass from elsewhere}). Since the original optimal partial transport maps are the identity on the first set, and must move at least some mass on the second, this modified candidate under the same mappings ends up providing a strictly smaller value in $F$ than $\Bary[i]$, contradicting that $\Bary[i]$ was a minimizer; this is Claim 3 which finishes the proof.

\end{enumerate}

\subsection{A preliminary lemma on convexity}
The following lemma is required for the proof of Theorem \ref{thm: partial barycenter is unique}.  It shows that for a fixed $\Dominating$, the functional $\PartialWassSq{\Dominating}{\cdot}$ is convex with respect to linear interpolation. Additionally, we show that non-strict convexity along a segment connecting two measures $\Bary[0]$ and $\Bary[1]$ implies some structure of the optimal mappings pushing $\Bary[i]$ forward to a submeasure of $\Dominating$: namely that the optimal mappings must match on the support of the ``pointwise minimum'' of $\Bary[0]$ and $\Bary[1]$, while both mappings must be the identity mapping when this ``pointwise minimum'' fails to saturate.
\begin{lem}\label{lem: convexity claims}
 Let $\Dominating$ be an absolutely continuous measure with $\mass{\Dominating}\geq \m$, let $\Bary[0]$, $\Bary[1]\in\Constraint$, and define $\Bary[t]:=(1-t)\Bary[0]+t\Bary[1]$. Then 
 \begin{align*}
 \PartialWassSq{\Dominating}{\Bary[t]}\leq (1-t)\PartialWassSq{\Dominating}{\Bary[0]}+t\PartialWassSq{\Dominating}{\Bary[1]}
 \end{align*}
  for all $t\in[0,1]$.
  Now suppose equality holds for all $t\in[0, 1]$, and for $i=0$ or $1$, let $\DualMapping[i]$ be any measurable mapping satisfying $\RevMappingCoupling[i]_\#\Bary[i]\in \Opt[m]{\Dominating}{\Bary[i]}$.  Also, let $\BaryCommon$ be the measure with density $\BaryCommonDens:=\min{\{\BaryDens[0],\BaryDens[1]\}}$, where $\BaryDens[i]$ is the density of $\Bary[i]$.  
  Then we have that 
  \begin{align}
  \DualMapping[0](y)&=\DualMapping[1](y)\text{ a.e. on }\{\BaryCommonDens>0\}  \label{eqn: mappings match on pointwise minimum},\\
  \DualMapping[i](y)&=y\text{ a.e. on }\{\BaryCommonDens<\BaryDens[i]\}.\label{eqn: identity on disjoint part}
  \end{align}
 \end{lem}

\begin{proof}
 For $i=0$, $1$ suppose that $\Paired[i]\leq \Dominating$ with total mass $\m$, and $\Coupling[i]\in\Opt{\Paired[i]}{\Bary[i]}$ satisfy $\Cost[2]{\Coupling[i]}=\PartialWassSq{\Dominating}{\Bary[i]}$; in words, $\Coupling[i]$ is optimal in the two marginal partial transport problem \eqref{eqn: OTm}, with marginals $\Dominating$ and $\Bary[i]$, and  $\Paired[i]$ and is its left active submeasure\footnote{Note that the ``$\leq$'' in $\Paired[i]$ is meant to remind the reader that $\Paired[i]$ is a submeasure of $\Dominating$.}.

 For any $t\in[0,1]$, it is clear that for $\Coupling[t]:=(1-t)\Coupling[0]+t\Coupling[1]$ we have
\begin{align*}
 (\Proj{1})_\#\Coupling[t]&=(1-t)\Paired[0]+t\Paired[1]\leq\Dominating,\\
 (\Proj{2})_\#\Coupling[t]&=(1-t)\Bary[0]+t\Bary[1]=\Bary[t],
\end{align*}
thus we easily see that
\begin{align}\label{eqn: convexity}
 \PartialWassSq{\Dominating}{\Bary[t]}\leq \Cost[2]{\Coupling[t]}&=(1-t)\Cost[2]{\Coupling[0]}+t\Cost[2]{\Coupling[1]}\nonumber\\
 &=(1-t)\PartialWassSq{\Dominating}{\Bary[0]}+t\PartialWassSq{\Dominating}{\Bary[1]}.
\end{align}
 
 We now turn to the proof of ~\eqref{eqn: mappings match on pointwise minimum} and ~\eqref{eqn: identity on disjoint part}.  Suppose that there is non-strict convexity, i.e. 
\begin{align*}
 \PartialWassSq{\Dominating}{\Bary[t]}= (1-t)\PartialWassSq{\Dominating}{\Bary[0]}+t\PartialWassSq{\Dominating}{\Bary[1]},\qquad\forall t\in[0, 1].
\end{align*}
We then have equality throughout \eqref{eqn: convexity}; in particular $\PartialWassSq{\Dominating}{\Bary[t]}=\Cost[2]{\Coupling[t]}$ and hence 
\begin{equation}\label{eqn: optimal time t}
\Coupling[t]\in\Opt[m]{\Dominating}{\Bary[t]}. 
\end{equation}
Now note that for $i=0$ and $1$, $\Bary[i]$ and $\Paired[i]$ are absolutely continuous; we denote their densities by $\BaryDens[i]$ and $\PairedDens[i]$ respectively. Also 
we may apply the classical result of Brenier (see~\cite{Bre91}) to see there exist a.e. defined mappings 
\begin{align*}
\Mapping[i]:& \spt{\Paired[i]}\to\spt{\Bary[i]},\\
\DualMapping[i]:& \spt{\Bary[i]}\to\spt{\Paired[i]}
\end{align*}
 such that $\Mapping[i]_\#\Paired[i]=\Bary[i]$, $\MappingCoupling[i]_\#\Paired[i]=\Coupling[i]=\RevMappingCoupling[i]_\#\Bary[i]$, and $\DualMapping[i]=\MappingInv[i]$ a.e. on $\spt{\Bary[i]}$.
 
 We will now show~\eqref{eqn: mappings match on pointwise minimum}. Indeed, by \eqref{eqn: optimal time t}
 for any $t\in[0, 1]$ we have that $\Coupling[t]\in\Opt{\Paired[t]}{\Bary[t]}$, where $\Paired[t]:=(1-t)\Paired[0]+t\Paired[1]$, and in particular, by Brenier's theorem, $\Coupling[1/2]$ is concentrated on the graph of an  a.e. defined mapping  $\DualMapping[1/2]:  \spt{\Bary[1/2]\to\spt{\Paired[1/2]}}$;  this means that for almost every $y \in \spt{\Bary[1/2]}$, there is a \emph{unique} $x$ (namely $x =\DualMapping[1/2](y)$) such that $(x,y) \in \spt\Coupling[1/2]$.  On the other hand,   $\Coupling[1/2]$ is clearly supported on the union of the graphs of  $\DualMapping[0]$ and $\DualMapping[1]$;  therefore, for almost every $y \in \{\BaryCommonDens>0\}$, \emph{both} $(\DualMapping[0](y), y)$ and $(\DualMapping[1](y), y)$ belong to  $\spt\Coupling[1/2]$.  This is possible only if, for almost every $y \in \{\BaryCommonDens>0\}$, we have $\DualMapping[0](y)=\DualMapping[1](y) =\DualMapping[1/2](y)$.  This implies  \eqref{eqn: mappings match on pointwise minimum}.

 
We next work toward~\eqref{eqn: identity on disjoint part}. As a result of~\eqref{eqn: mappings match on pointwise minimum}, we can unambiguously define 
\begin{align*}
 \CommonCoupling:=\RevMappingCoupling[0]_\#\BaryCommon=\RevMappingCoupling[1]_\#\BaryCommon
\end{align*}
with mass $\BaryCommonMass:=\mass{\CommonCoupling}\leq \m$. Here we claim that 
\begin{align*}
( \Proj{1})_\#\CommonCoupling=\PairedCommon
\end{align*}
where $\PairedCommon$ is the absolutely continuous measure with density $\PairedCommonDens:=\min\{\PairedDens[0],\ \PairedDens[1]\}$. To see this, first note it is clear that 
\begin{align*}
( \Proj{1})_\#\CommonCoupling=\DualMapping[0]_\#\BaryCommon=\DualMapping[1]_\#\BaryCommon\leq\PairedCommon.
\end{align*}
 Next, 
 we can apply the arguments leading up to~\eqref{eqn: mappings match on pointwise minimum} with $\Mapping[i]$ replacing $\DualMapping[i]$ to find that for a.e. $x\in\{\PairedCommonDens>0\}$ we have $\Mapping[0](x)=\Mapping[1](x)$, hence 
\begin{align*}
\Mapping[0]_\#\PairedCommon=\Mapping[1]_\#\PairedCommon&\leq \BaryCommon\\
\implies\PairedCommon=\DualMapping[0]_\#\Mapping[0]_\#\PairedCommon&\leq \DualMapping[0]_\#\BaryCommon=( \Proj{1})_\#\CommonCoupling,
\end{align*}
finishing the claim.

Next we claim that for any $t\in[0, 1]$ we have
\begin{align}
 (1-t)(\Coupling[0]-\CommonCoupling)&\in\Opt[(1-t)(\m-\BaryCommonMass)]{\Paired[0]-\PairedCommon}{(1-t)(\Bary[0]-\BaryCommon)},\notag\\
 t(\Coupling[1]-\CommonCoupling)&\in\Opt[t(\m-\BaryCommonMass)]{\Paired[1]-\PairedCommon}{t(\Bary[1]-\BaryCommon)}.\label{eqn: partially optimal}
\end{align}
We will prove the claim by contradiction.  Assuming that it fails, we will construct a measure in $\PartialCouples{\Dominating}{\Bary[t]} $ with total mass $m$ and a lower cost than $\Coupling[t]$, contradicting  $\Coupling[t]\in\Opt[m]{\Dominating}{\Bary[t]}$.

Suppose by contradiction that the claim fails; then there exist
\begin{align*}
 \BarCoupling[0]&\in\Opt[(1-t)(\m-\BaryCommonMass)]{\Paired[0]-\PairedCommon}{(1-t)(\Bary[0]-\BaryCommon)},\\
 \BarCoupling[1]&\in\Opt[t(\m-\BaryCommonMass)]{\Paired[1]-\PairedCommon}{t(\Bary[1]-\BaryCommon)}
\end{align*}
with
\begin{align*}
\Cost[2]{\BarCoupling[0]}+\Cost[2]{\BarCoupling[1]}&<\Cost[2]{(1-t)(\Coupling[0]-\CommonCoupling)}+\Cost[2]{t(\Coupling[1]-\CommonCoupling)}\\
&=\Cost[2]{\Coupling[t]-\CommonCoupling},
\end{align*}
(where we have used linearity of $\Cost[2]{\cdot}$), which then implies
\begin{equation}
\Cost[2]{\BarCoupling[0]+\BarCoupling[1]+\CommonCoupling} < \Cost[2]{\Coupling[t] }\label{eqn: strictly smaller}.
\end{equation}
Now note that 
\begin{align*}
 \mass{\BarCoupling[0]+\BarCoupling[1]+\CommonCoupling}=(1-t)(\m-\BaryCommonMass)+t(\m-\BaryCommonMass)+\BaryCommonMass=\m.
\end{align*}
Also,
\begin{align*}
 (\Proj{1})_\#(\BarCoupling[0]+\BarCoupling[1]+\CommonCoupling)&\leq (\Paired[0]-\PairedCommon)+(\Paired[1]-\PairedCommon)+\PairedCommon\\
 &=\Paired[0]+\Paired[1]-\PairedCommon\\
 &=(\PairedDens[0]+\PairedDens[1]-\PairedCommonDens)dx\\
 &=\max{\{\PairedDens[0],\PairedDens[1]\}}dx\\
 &\leq \Dominating.
 \end{align*}
On the other hand, the second marginal satisfies
\begin{align*}
 (\Proj{2})_\#(\BarCoupling[0]+\BarCoupling[1]+\CommonCoupling)=(1-t)(\Bary[0]-\BaryCommon)+t(\Bary[1]-\BaryCommon)+\BaryCommon=\Bary[t],
 \end{align*}
 hence combined with~\eqref{eqn: strictly smaller} this would contradict that $\Coupling[t]\in\Opt[m]{\Dominating}{\Bary[t]}$, and we obtain the claim \eqref{eqn: partially optimal}.
 
 Finally for $t \in(0,1)$, note that $(1-t)(\PairedDens[0]-\PairedCommonDens) < \PairedDens[0]-\PairedCommonDens$ wherever $\PairedDens[0]-\PairedCommonDens >0$ (and thus a.e. on $\spt(\Paired[0]-\PairedCommon)$). Since $(\Proj{1})_\#\left((1-t)(\Coupling[0]-\CommonCoupling)\right)=(1-t)(\PairedDens[0]-\PairedCommonDens)dx$, 
%
%
%
%
by combining~\cite[Theorem 2.6 (2.5)]{Fig10a} with~\eqref{eqn: partially optimal} we see that $(1-t)(\Coupling[0]-\CommonCoupling)$  is supported on the diagonal $\{x=y\}\subseteq \mathbb{R}^{2n}$; this immediately implies that $\Coupling[0]-\CommonCoupling=\RevMappingCoupling[0]_\#(\Bary[0]-\BaryCommon)$ is as well. This means that for $(\Bary[0]-\BaryCommon)$-a.e. $y$ we must have $\DualMapping[0](y)=y$, and with a symmetric argument applied to $\Coupling[1]-\CommonCoupling$ we obtain  ~\eqref{eqn: identity on disjoint part}.
 \end{proof}
%
\subsection{Proof of uniqueness}
Before presenting the proof of Theorem~\ref{thm: partial barycenter is unique}, we remark that the argument contains a fair bit of notation.  To  make the proof more accessible, we provide in an appendix (Appendix \ref{sect: notation}) a summary of the main notation.  It might be convenient for the reader to keep the summary handy while following the proof.  In addition, several technical aspects of the argument, identified as claims within the body of the proof, have been deferred to another appendix (Appendix \ref{sect: proof of claims}), to avoid interrupting the main flow of the argument.

We are now ready to prove Theorem~\ref{thm: partial barycenter is unique}.
\begin{proof}[{Proof of Theorem~\ref{thm: partial barycenter is unique}}]
Fix $\m$ as in the statement of the theorem, and suppose by contradiction that $\Bary[0]\neq\Bary[1]$ are both minimizers in~\eqref{eqn: BCm}; again by Lemma~\ref{lem: absolute continuity of BCm minimizers} we  have that both $\Bary[0]$, $\Bary[1]\in\Constraint$. We will construct a $\BaryBar\in\Constraint$ which achieves a lower value than the minimal value attained by $\Bary[0]$ and $\Bary[1]$ in~\eqref{eqn: BCm}, contradicting their minimality.

Since each summand in $\Func$ is convex under linear interpolation by Lemma~\ref{lem: convexity claims}, so is $\Func$. In particular, as $\Functional{\Bary[0]}=\Functional{\Bary[1]}$ and both achieve the minimum value, we see that $\Functional{\Bary[t]}=(1-t)\Functional{\Bary[0]}+t\Functional{\Bary[1]}$ for all $t\in[0, 1]$ with $\Bary[t]$ defined as in Lemma~\ref{lem: convexity claims}, which in turn implies 
\begin{align*}
 \PartialWassSq{\Dominating[j]}{\Bary[t]}= (1-t)\PartialWassSq{\Dominating[j]}{\Bary[0]}+t\PartialWassSq{\Dominating[j]}{\Bary[1]}
\end{align*}
 for all $t\in[0,1]$ and $1\leq j\leq \Number$. For each $i=0$, $1$ and $1\leq j\leq \Number$ we again obtain an a.e. defined collection of optimal partial transport maps $\MultiDualMapping[i]{j}: \spt{\Bary[i]}\to \spt{\Dominating[j]}$ such that $\RevMultiMappingCoupling[i]{j}_\#\Bary[i]\in\Opt[\m]{\Dominating[j]}{\Bary[i]}$, and $\PairedMulti[i]{j}:=\Paren{\MultiDualMapping[i]{j}}_\#\Bary[i]\leq \Dominating[j]$. We will extend each $\MultiDualMapping[i]{j}$ to all of $\R^n$ by taking it to be the identity mapping where it is not defined (in particular, on $\R^n\setminus{\spt\Bary[i]}$). With this extension, by using Lemma~\ref{lem: convexity claims} \eqref{eqn: mappings match on pointwise minimum} and \eqref{eqn: identity on disjoint part} we can see that for every $1\leq j\leq \Number$,
\begin{align*}
 \MultiDualMapping[0]{j}(y)=\MultiDualMapping[1]{j}(y),\ \text{a.e. }y\in\R^n,
\end{align*}
 and additionally, 
 \begin{align}\label{eqn: identity map}
 \MultiDualMapping[i]{j}(y)= y,\ \text{a.e. }y\in\StaySet\cup\{\BaryDens[i]=0\}
 \end{align}
 where
%
\begin{align*}
 \StaySet:=\{y\in\R^n\mid \BaryDens[0](y)\neq\BaryDens[1](y)\},
\end{align*}
and $\BaryDens[i]$ is the density of $\Bary[i]$. We also note here that $\MultiDualMapping[i]{j}$ is injective a.e. on $\spt{\Bary[i]}$ by the Brenier-McCann theorem \cite{Bre87,Bre91, McC95} (as  $\Dominating[j]$ is absolutely continuous).  Finally, it is easy to see that the inverse image of any set under $\MultiDualMapping[i]{j}$ differs from its inverse image under the map before extension only by a $\Bary[i]$-null set, in particular we still have $\PairedMulti[i]{j}=\Paren{\MultiDualMapping[i]{j}}_\#\Bary[i]$ after this extension is made.

By these observations,  for any measurable $E\subset \StaySet$ we have (up to negligible sets) $E\subset \MultiDualMappingInv[i]{j}(E)$. Hence,
\begin{align*}
 \Dominating[j](E)&\geq \Paren{\MultiDualMapping[i]{j}}_\#\Bary[i](E)\\
 &=\Bary[i](\MultiDualMappingInv[i]{j}(E))\geq \Bary[i](E).
\end{align*}
Therefore,

$ \BaryDens[i] \leq \DominatingDens[j]$ a.e. on $\StaySet$, for each $j=1,2,...,N$, and so

\begin{align}\label{eqn: barycenter dominated by pointwise minimum on stay set}
 \BaryDens[i]\leq\DominatingCommonDens
\end{align}
a.e. on $\StaySet$.

Now, by the assumption $\mass{\Bary[i]}=\m>\mass{\DominatingCommon}$, there exists a set $\MinSmaller$ of strictly positive Lebesgue measure on which $\DominatingCommonDens<\BaryDens[i]$; by~\eqref{eqn: barycenter dominated by pointwise minimum on stay set} we must have $\MinSmaller\subset\{\BaryDens[i]>0\}\setminus\StaySet$,  after possibly discarding a null set, and as a result we can assume $\DominatingCommonDens<\BaryDens[0]=\BaryDens[1]$ on $\MinSmaller$.

We now make the following claim:

\textbf{Claim 1:} \textit{for either $i=0$ or $1$, there exists a Borel set $\MinBigger\subset\StaySet$ with strictly positive Lebesgue measure on which $\DominatingCommonDens>\BaryDens[i]$, which also satisfies
\begin{align}
 \int_{\MinSmaller}(\BaryDens[i]-\DominatingCommonDens)dx &=\int_{\MinBigger}(\DominatingCommonDens-\BaryDens[i])dx>0\label{eqn: bigger smaller sets}
\end{align}
and
\begin{align}\label{eqn: doesn't get mass from elsewhere}
 \Leb{\MultiDualMapping[i]{j}(\R^n\setminus\MinBigger)\cap\MinBigger}=0
\end{align}
for all $1\leq j\leq\Number$. }

To avoid interrupting the main flow of the argument, the proof of Claim 1 is deferred to Appendix \ref{sect: proof of claims}.

Now, note that clearly $ \MinSmaller\cap\MinBigger=\emptyset$. Let us define 
\begin{align*}
 \BaryBar:=\Bary[i]+(\DominatingCommonDens-\BaryDens[i])\Indicator_{\MinSmaller\cup\MinBigger}dx,
\end{align*}
which is a positive measure.
It is clear $\BaryBar$ is absolutely continuous, and by~\eqref{eqn: bigger smaller sets} and  the disjointness of $\MinSmaller$ and $\MinBigger$, we have
\begin{align*}
 \mass{\BaryBar}&=\mass{\Bary[i]}+\int_{\MinBigger}(\DominatingCommonDens-\BaryDens[i])dx-\int_{\MinSmaller}(\BaryDens[i]-\DominatingCommonDens)dx\\
 &=\m,
\end{align*}
i.e. $\BaryBar\in\Constraint$.

Next we claim the following (the proof can be found in   Appendix \ref{sect: proof of claims}).

\textbf{Claim 2:}\textit{ $\Paren{\MultiDualMapping[i]{j}}_\#{\BaryBar}\leq\Dominating[j]$ for $1\leq j\leq \Number$. }

 In particular, this shows that for each $i$ and $j$, the measure $\RevMultiMappingCoupling[i]{j}_\#\BaryBar$ is an admissible competitor in $\PartialWassSq{\Dominating[j]}{\BaryBar}$.
 
 Finally, we make one more claim (again proved in  Appendix \ref{sect: proof of claims}).

\textbf{Claim 3:}
\textit{
\begin{align}\label{eqn: new map too cheap}
 \sum_{j=1}^\Number{\Cost[2]{\RevMultiMappingCoupling[i]{j}_\#\BaryBar}}<\sum_{j=1}^\Number{\Cost[2]{\RevMultiMappingCoupling[i]{j}_\#\Bary[i]}}.
\end{align}
}

Since $\Functional{\BaryBar}\leq\sum_{j=1}^\Number{\Cost[2]{\RevMultiMappingCoupling[i]{j}_\#\BaryBar}}$ and $\Functional{\Bary[i]}=\sum_{j=1}^\Number{\Cost[2]{\RevMultiMappingCoupling[i]{j}_\#\Bary[i]}}$, this last claim contradicts the fact that $\Bary[i]$ is a minimizer in~\eqref{eqn: BCm}, and finishes the proof.

\end{proof}

\section{Other properties of the multi-marginal optimal partial transport problem and the partial barycenter problem}\label{sec: counterexamples}
In this section, we will discuss other analytic properties of minimizers in~\eqref{eqn: MMm}. We begin with a counterexample to the ``monotonicity property'', in contrast to the two marginal case of~\eqref{eqn: OTm} (see~\cite[Theorem 3.4]{CM10} and~\cite[Remark 3.4]{Fig10a}).
\begin{prop}\label{prop: nonmonotone}(\textbf{Partial barycenters and active submarginals may not depend monotontically on $m$.})
There exist measures $\Dominating[1], \Dominating[2],\Dominating[3]$ on $\mathbb{R}$ 
for which:
\begin{enumerate}
 \item The mapping $m \mapsto\Bary[\m]$, where $\Bary[\m]$ is the minimizer of $\sum_{j=1}^3\PartialWassSq{\Dominating[j]}{\Bary}$ over $\Constraint$ is not monotone, in the sense that there exists  $0<m < \mbar < \min_{j=1,2,3}\mass{\Dominating[j]}$ for which $\Bary[\m]$ is not a submeasure of $\Bary[\mbar]$.
\item The mapping $m \mapsto \PairedMulti[\m]{3}:=\Paren{\Proj{3}}_\#\Paren{\MMCoupling[m]}$, where  $\MMCoupling[\m]\in \OptMulti[m]{\Dominating[1], \Dominating[2], \Dominating[3]}$, is not monotone, in the sense that there exists  $0<m < \mbar < \min_{j=1,2,3}\mass{\Dominating[j]}$ for which $\PairedMulti[\m]{3}$ is not a submeasure of $\PairedMulti[\mbar]{3}$. 

\end{enumerate}

\end{prop}
In fact, as we will see in the proof below, even more is true; the barycenter of the first two measures does not depend monotonically on $m$.

It may be helpful to see Figures~\ref{figure: continuous dominating measures} and~\ref{figure: continuous active measures} below while following the proof.
\begin{figure}[h]
  \centering
    \includegraphics[width=\textwidth]{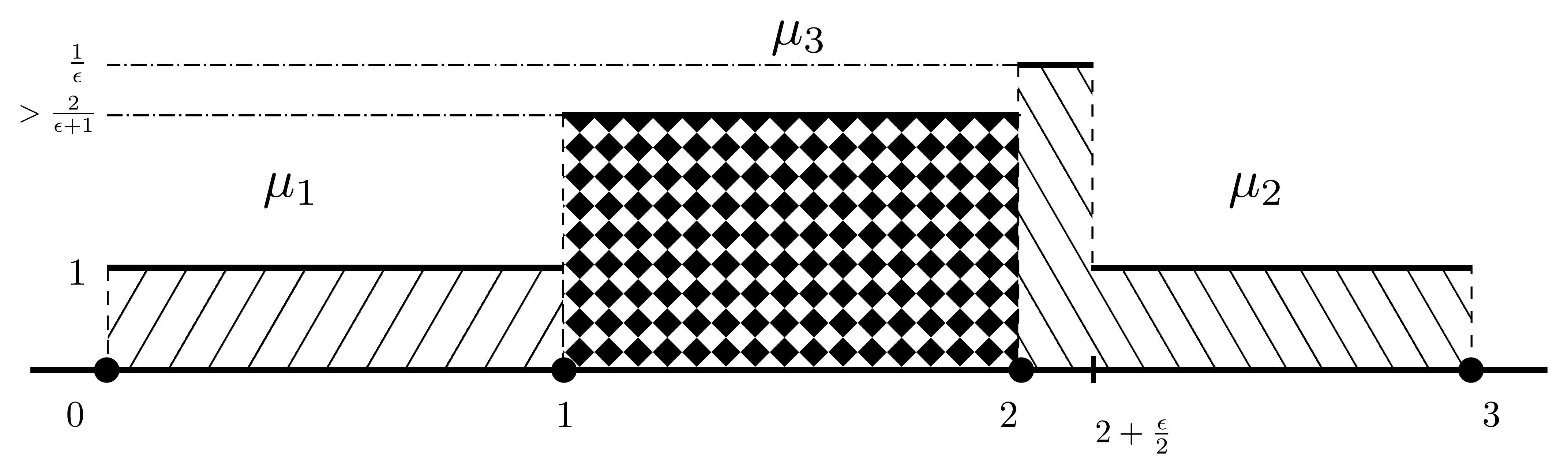}
     \caption{Here we have shown graphically the three dominating measures $\Dominating[1]$, $\Dominating[2]$, and $\Dominating[3]$ in Proposition~\ref{prop: nonmonotone}.}\label{figure: continuous dominating measures}
  \centering
    \includegraphics[width=\textwidth]{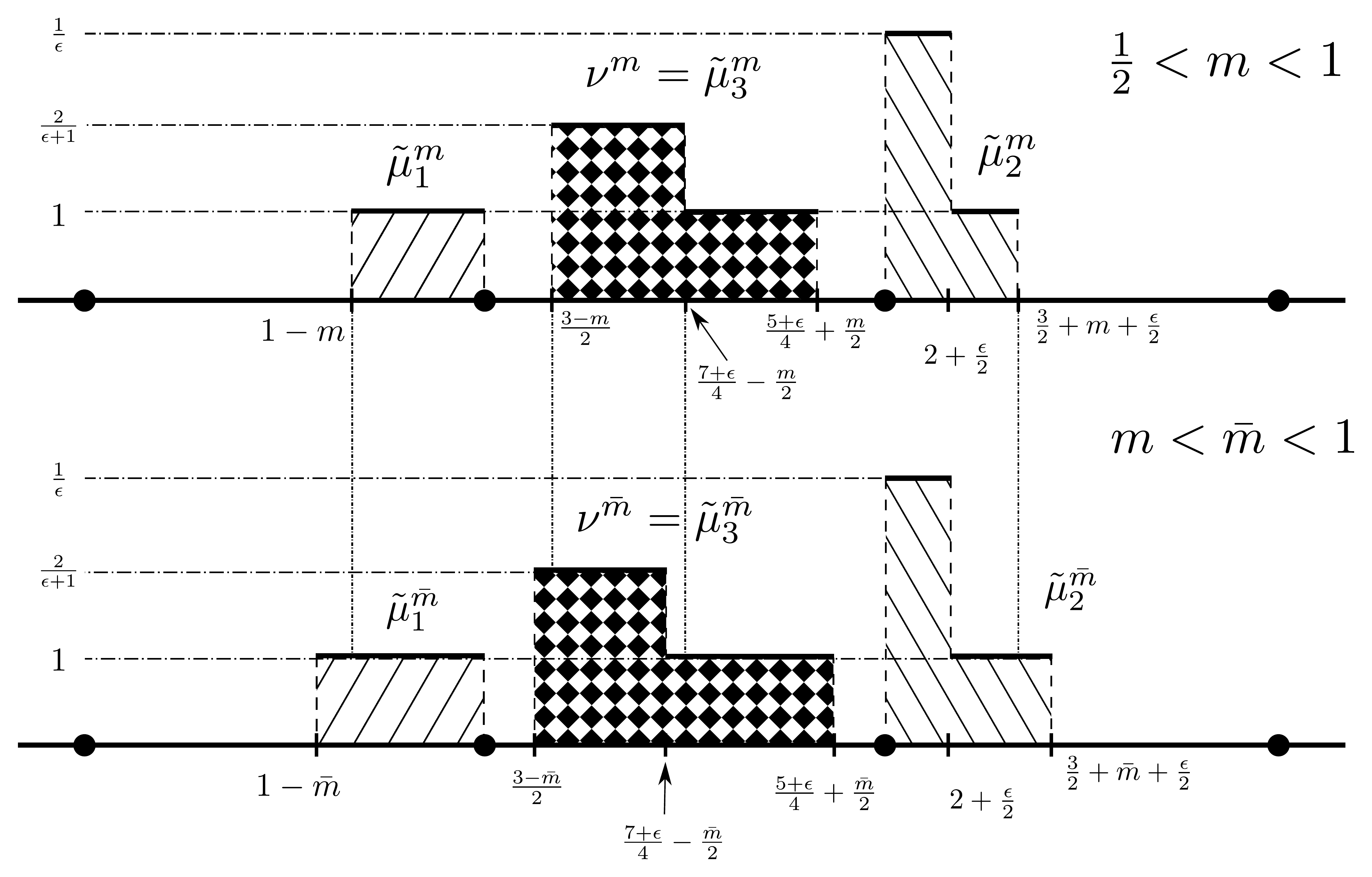}
     \caption{This figure illustrates the active submeasures in Proposition~\ref{prop: nonmonotone} for two different values $\frac{1}{2}<\m <\mbar<1$. Note that in each case, the partial barycenter is equal to the third active submeasure, which is not monotone  in $\m$ (as demonstrated by the dashed vertical lines).}\label{figure: continuous active measures}
\end{figure}
\begin{proof}
Take $\Dominating[1]$ to be uniform measure on $[0,1]$ with density $1$ and take $\Dominating[2]$ to be absolutely continuous, supported on $[2,3]$ with density given by
\begin{align*}
\DominatingDens[2](x) &=
\begin{cases}
 \frac{1}{\epsilon} &\text{ on }[2, 2+ \frac{\epsilon}{2}],\\
 1 &\text{ on }[2+ \frac{\epsilon}{2},3].
 \end{cases}
\end{align*}

First, consider the optimal partial transport problem between the two marginals $\Dominating[1]$ and $\Dominating[2]$.  For $\frac{1}{2}<\m<1$, it is an easy consequence of \cite[Corollary 2.4]{CM10} combined with \cite[Theorem 2.6]{Fig10a} that if $\TwoCoupling[m]\in\Opt[m]{\Dominating[1]}{\Dominating[2]}$ then the active submeasures  $\PairedMulti[\m]{1}:=\Paren{\Proj{1}}_\#\TwoCoupling[m]$ and $\PairedMulti[\m]{2}:=\Paren{\Proj{2}}_\#\TwoCoupling[\m]$ are the measures $\Dominating[1]$ and $\Dominating[2]$, restricted to the intervals $[1-m,1]$ and $[2,\frac{3}{2}+\frac{\epsilon}{2}+m]$, respectively (these are the ``right most'' piece of the first measure and the ``left most'' piece of the second).  In addition we can write $\TwoCoupling[m]=\MultiMappingCoupling[\m]{}_\#\PairedMulti[\m]{1}$, where the optimal map $\MultiMapping[\m]{}$ between the two active submeasures is the unique increasing map pushing $\PairedMulti[\m]{1}$ forward to $\PairedMulti[\m]{2}$, given by:

\begin{align*}
\MultiMapping[\m]{}(x) &=
\begin{cases}
 \epsilon x+2-\epsilon(1-m)& \text{ on }[1-m, \frac{3}{2}-m],\\
x+\frac{1+\epsilon}{2}+m& \text{ on }[\frac{3}{2}-m,1].
\end{cases}
\end{align*}
By Proposition~\ref{prop: BCm=MMm}, the partial barycenter $\Bary[\m]:=\Paren{\frac{x+\MultiMapping[\m]{}(x)}{2}}_\#\PairedMulti[\m]{1}$ minimizes $\sum_{j=1}^2\PartialWassSq{\Dominating[j]}{\Bary}$ over $\Constraint$; thus $\Bary[\m]$ is supported on $[\frac{3-m}{2}, \frac{5+\epsilon}{4} +\frac{m}{2}]$, with a density given by
\begin{align*}
\BaryDens[\m](x) &=
\begin{cases}
 \frac{2}{\epsilon+1}& \text{ on }[\frac{3-m}{2},\frac{7+\epsilon}{4}-\frac{m}{2}],\\
 1& \text{ on }[\frac{7+\epsilon}{4}-\frac{m}{2},\frac{5+\epsilon}{4}+\frac{m}{2}].
 \end{cases}
\end{align*}
In particular, note that the partial barycenter of $\Dominating[1]$ and $\Dominating[2]$ is not monotone  with respect to the parameter $\m$; the location of the jump in $\BaryDens[\m]$ moves to the left as $\m$ increases, hence $\Bary[\m]\not\leq\Bary[\mbar]$ when $\frac{1}{2}<\m < \mbar<1$.

Now, take $\Dominating[3]$ to be uniform on $[1,2]$, with density $\DominatingDens[3]>\frac{2}{\epsilon+1}$.  Then each $\nu^m\leq\Dominating[3]$, and it is straightforward to see that $\Bary[\m]$ minimizes $\Bary\mapsto\sum_{j=1}^3\PartialWassSq{\Dominating[j]}{\Bary}$ over $\Constraint$, 
(as $\Bary[\m]$ minimizes $\Bary \mapsto \PartialWassSq{\Dominating[1]}{\Bary}+\PartialWassSq{\Dominating[2]}{\Bary}$ while $ \PartialWassSq{\Dominating[3]}{\Bary[m]}=0$, it clearly minimizes their sum).

As the active submeasure $\PairedMulti[\m]{3}$ corresponding to $\Dominating[3]$ is precisely $\Bary[\m]$, this shows that the active submeasures are not monotone  with respect to $\m$ either.
\end{proof}

\begin{rmk}\label{classcouplingnotoptimal} The example in the preceding proof also implies that the naive multi-marginal analogue of \cite[Proposition 2.4]{Fig10a} fails.  

The analogous statement would be the following: if $\MMCoupling$ solves~\eqref{eqn: MMm} with marginals $\Dominating[1], \ldots, \Dominating[\Number]$ and $\PairedMulti{j}:=\Paren{\Proj{j}}_\#\MMCoupling$, then 
\begin{align*}
 \BarMMCoupling:=\MMCoupling+\Paren{\bigotimes_{j=1}^\Number\Id}_\#\Paren{\sum_{j=1}^\Number(\Dominating[j]-\PairedMulti{j})}
\end{align*}
solves~\eqref{eqn: MM} with marginals $\Dominating[1]+\sum_{j\neq 1}(\Dominating[j]-\PairedMulti{j}),\ldots, \Dominating[\Number]+\sum_{j\neq \Number}(\Dominating[j]-\PairedMulti{j})$. 

However, we can see that already in the case $\Number=3$, this statement does not hold for the example given above in Proposition~\ref{prop: nonmonotone}. Indeed, note that for $x \in[\frac{7+\epsilon}{4}-\frac{m}{2},\frac{5+\epsilon}{4}+\frac{m}{2}]$ we have $\PairedMultiDens{3}(x) =1 < \frac{2}{\epsilon+1}< \DominatingDens[3](x)$;  therefore the density $\DominatingDens[3]-\PairedMultiDens{3}$ of $\Dominating[3] - \PairedMulti{3}$ is strictly positive on this interval. As a result, for each $x_3 \in [\frac{7+\epsilon}{4}-\frac{m}{2},\frac{5+\epsilon}{4}+\frac{m}{2}]$ the support of $\BarMMCoupling$  includes $(x_3,x_3,x_3)$   (via  $\Paren{\Id \times \Id \times \Id}_\#\Paren{\Dominating[3] - \PairedMulti{3}}$) as well as points of the form 
$(x_1,x_2,x_3)$, where $x_j \in spt(\PairedMultiDens{j})$ for $j=1$, $2$ (via $\MMCoupling$).
In particular, $ \BarMMCoupling$  is not concentrated on a graph over the third marginal, and so, by \cite[Theorem 2.1]{GS98} $\BarMMCoupling$ cannot be optimal in~\eqref{eqn: MM}.

As Proposition 2.4 of \cite{Fig10a} plays a key role in Figalli's proof of Theorem 2.10 there, this indicates that a direct application of the techniques in \cite{Fig10a} does not translate to the multi-marginal case.

\end{rmk}

\
The next example shows that, in contrast to a result of Caffarelli and McCann in the $\Number=2$ case, monotonicity of the active submeasures  with respect to $\m$ can fail even for discrete measures (compare~\cite[Proposition 3.1]{CM10}). 
\begin{figure}[h]
  \centering
    \includegraphics[width=.9\textwidth]{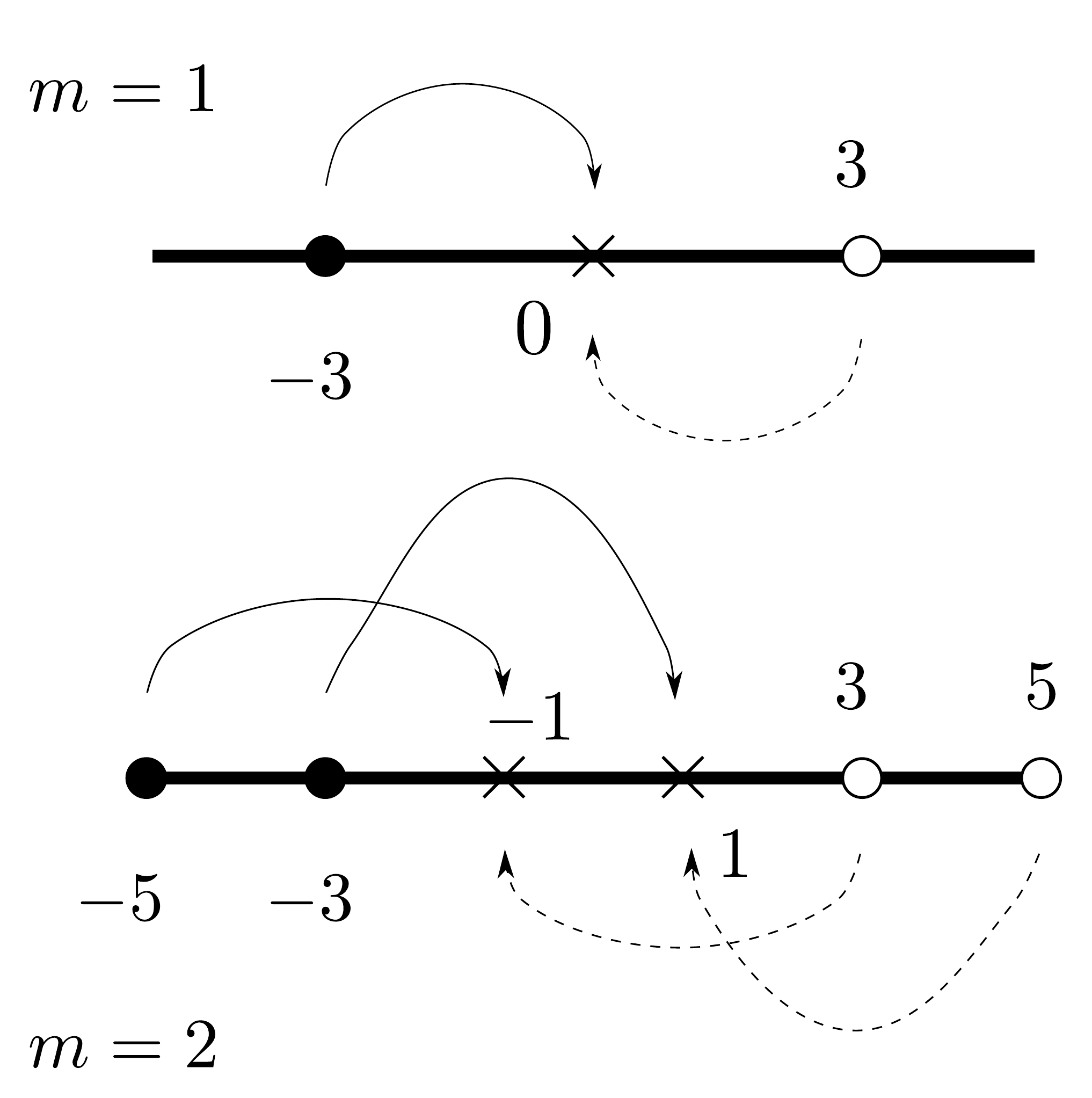}
     \caption{This figure illustrates the active submeasures in Example~\ref{ex: deltas} for $m=1$ and $m=2$. Filled in dots represent the support of the active submeasure of $\Dominating[1]$, empty dots the active submeasure of $\Dominating[2]$, and the crosses are the active submeasure of $\Dominating[3]$. Each submeasure is a sum of unit Dirac measures supported at the various points. Again in each case, the partial barycenter (whose couplings are illustrated by the solid and dotted arrows) is equal to the third active submeasure, which fails to be monotone  with respect to $\m$ (in fact, the support itself does not grow monotonically in $\m$).}\label{figure: discrete}
\end{figure}
\begin{ex}\label{ex: deltas}
Consider the real line $\R$ and take
\begin{align*}
\Dominating[1]&=\delta_{-5}+\delta_{-3},\\
\Dominating[2] &= \delta_{-1}+\delta_{0}+\delta_{1},\\
\Dominating[3] &=\delta_{3}+\delta_{5}.
\end{align*}
Taking $\m=1$, it is easy to see that the optimizer is $\delta_{(-3, 0, 3)}$ which couples the Dirac masses at $-3$, $0$, and $3$, while for $\m=2$, the optimizer is $\delta_{(-5, -1, 3)}+\delta_{(-3, 1, 5)}$ which couples the masses at $-5$, $-1$, and $3$; and $-3$, $1$, and $5$ respectively.  This shows that even the \emph{support} of the active submeasure of $\Dominating[2]$  may not grow monotonically with respect to $\m$.

 An example with absolutely continuous measures where the supports of the active submeasures are not monotone can be constructed by replacing the Dirac masses with uniform measure on small disjoint intervals; however it is not clear to us whether an example can be constructed in which the marginals are absolutely continuous with connected supports. 
\end{ex}

As the above two examples illustrate, when there are three or more marginals in~\eqref{eqn: MMm}, an optimal coupling may move mass away from a location where the active submeasure does not saturate the dominating measure: in Proposition~\ref{prop: nonmonotone} above, on $\spt{\Dominating[3]}$ we have $\PairedMultiDens{3}<\DominatingDens[3]$ (where $\PairedMultiDens{3}$ is the density of $\PairedMulti{3}$), yet under the optimal coupling \emph{none} of the mass of $\PairedMulti{3}$ remains in place. At first glance, this seems to be a sharp distinction from the two marginal case,~\cite[Theorem 2.6]{Fig10a}, however we now show there is an appropriate analogous statement for the multi-marginal case, in the form of Corollary~\ref{cor: average characterization}.

\begin{prop}\label{prop: massfilling}
 Suppose that $\MMCoupling\in\OptMulti[m]{\Dominating[1],\ldots,\Dominating[\Number]}$, and $\DominatingDens[j]$ is the density of $\Dominating[j]$. Then for a.e. $x_j$ in $\{\PairedMultiDens{j}<\DominatingDens[j]\}$, where $\PairedMultiDens{j}$ is the density of $\Paren{\Proj{j}}_\#\MMCoupling$,
\begin{align*}
 x_j=\frac{1}{\Number-1}\sum_{k\neq j}x_k
\end{align*}
where $\{x_k\}_{k \neq j}$ is the unique collection of points such that $\Paren{x_1,\ldots, x_\Number}\in\spt{\MMCoupling}$.
\end{prop}
\begin{proof}
Let us write 
\begin{align*}
\TwoCoupling[j] :&= \Paren{\Proj{j}\times\Avg}_{\#}\MMCoupling,\\
\PairedMulti{j} :&= \Paren{\Proj{j}}_\#\MMCoupling
\end{align*}
(recall the definition of the mapping $\Avg$ is given by~\eqref{eqn: average}). By Proposition~\ref{prop: BCm=MMm}, we have $\TwoCoupling[j]\in \Opt[\m]{\Dominating[j]}{\Avg_\#\MMCoupling}$, thus by the absolute continuity of $\Dominating[j]$ there exists a mapping $\MultiMapping{j}$ defined a.e. on $\spt{\PairedMulti{j}}$ such that $\TwoCoupling[j]=\MultiMappingCoupling{j}_\#\PairedMulti{j}$. By~\cite[Theorem 2.6]{Fig10a}, for a.e. $x_j\in\{\PairedMultiDens{j}<\DominatingDens[j]\}$ we have 
\begin{align}\label{eqn: identity mapping}
\MultiMapping{j}(x_j)=x_j,
\end{align}
 note that for a.e. $x_j\in\spt{\PairedMulti{j}}$, 
\begin{align}\label{eqn: injective mapping}
(x_j, y)\in\spt{\TwoCoupling[j]}\implies y=\MultiMapping{j}(x_j).
\end{align}
 On the other hand, by~\cite[Theorem 2.1, Corollary 2.2]{GS98} for a.e. $x_j$, there exists a unique set of points $\{x_k\}_{k \neq j}$ such that $\Paren{x_1,\ldots, x_\Number}\in\spt{\MMCoupling}$. Thus for a.e. $x_j\in\spt{\PairedMulti{j}}$, we see that
\begin{align*}
\Paren{x_j, \Avg(x_1,\ldots, x_\Number)}&\in \spt{\TwoCoupling[j]}
\end{align*}
which combined with~\eqref{eqn: injective mapping} implies
\begin{align*}
\Avg(x_1,\ldots, x_\Number)=\MultiMapping{j}(x_j).
\end{align*}
Finally, combining this with~\eqref{eqn: identity mapping} we find that for a.e. $x_j\in\{\PairedMultiDens{j}<\DominatingDens[j]\}$, there exists unique $\{x_k\}_{k \neq j}$ such that $\Paren{x_1,\ldots, x_\Number}\in\spt{\MMCoupling}$ and
\begin{align*}
 x_j=\MultiMapping{j}(x_j)=\Avg(x_1,\ldots, x_\Number)=\frac{1}{\Number}\sum_{k=1}^\Number x_k.
\end{align*}
By rearranging, we obtain the conclusion of the proposition.
\end{proof}
\begin{cor}\label{cor: average characterization}
Let $\MMCoupling$, $\DominatingDens[j]$, and $\PairedMultiDens{j}$ be as above. Also fix an integer $1\leq K<\Number$, and some sub-collection of indices $\Indices:=\{j_1,\ldots, j_{K}\}\subset\{1,\ldots, \Number\}$. Then for $\Paren{\R^n}^K$-Lebesgue a.e. $(x_{j_1},\ldots, x_{j_K})\in \prod_{j\in\Indices}\{\PairedMultiDens{j}<\DominatingDens[j]\}$ which can be completed to some $(x_1,\ldots, x_\Number)\in\spt{\MMCoupling}$, the following hold:
\begin{align}
x_{j_1}&=\ldots=x_{j_K}\label{eqn: all the same},\\
x_{j_1}&=\frac{1}{\Number-K}\sum_{k\not\in\Indices}x_k\label{eqn: averages}.
\end{align}
\end{cor}
\begin{proof}
The case $K=1$ and $\Number=2$ is exactly~\cite[Theorem 2.6]{Fig10a}, so let us assume that $\Number>2$. Clearly there is a $\Paren{\R^n}^K$-Lebesgue full measure subset of $\prod_{j\in\Indices}\{\PairedMultiDens{j}<\DominatingDens[j]\}$ on which Proposition~\ref{prop: massfilling} applies to every component; fix one such point $(x_{j_1},\ldots, x_{j_K})$ in that set and complete it to some $(x_1,\ldots, x_\Number)\in\spt{\MMCoupling}$. Then we can see, for example,  
\begin{align*}
x_{j_1}&=\frac{1}{\Number-1}\sum_{k\neq j_1}x_k,\\
x_{j_2}&=\frac{1}{\Number-1}\sum_{k\neq j_2}x_k,
\end{align*}
and by subtracting and rearranging (since $\Number>2$) we find that $x_{j_1}=x_{j_2}$. Proceeding as such for other indices in $\Indices$, we immediately obtain the claim~\eqref{eqn: all the same}.

Now, by another application of Proposition~\ref{prop: massfilling},
\begin{align*}
x_{j_1}&=\frac{1}{\Number-1}\sum_{k\neq j_1}x_k\\
\iff(\Number-1)x_{j_1}&=\Paren{(K-1)x_{j_1}+\sum_{k\not\in \Indices}x_k}\\
\iff (\Number-K)x_{j_1}&=\sum_{k\not\in \Indices}x_k,
\end{align*}
and we obtain~\eqref{eqn: averages}.
\end{proof}
In particular, if $K=\Number-1$ above, we recover an appropriate analogue of~\cite[Theorem 2.6]{Fig10a} in the multi-marginal case: if $\Number-1$ of the active submeasures do not saturate the original measures at an $\Number$-tuple in the optimal coupling, then all of the coupled points must be the same (up to a set of measure zero).
\begin{rmk}[Semiconcavity of the free boundary]\label{rmk: semiconcavity of the free boundary}
Lastly we remark that under certain conditions, we can obtain the semiconcavity of the free boundary ``for free'' simply by applying the theory of the two marginal case. Assume that each support $\spt{\Dominating[j]}$, $1\leq j\leq \Number$ is separated by a hyperplane from their Minkowski average. Note the support of any $\Bary$ that minimizes~\eqref{eqn: BCm} is contained in this Minkowski average by \cite[Proposition 4.2]{AC11}. Then, the marginal of any optimizer in~\eqref{eqn: MMm} can be thought of as the left marginal of~\eqref{eqn: OTm} with right marginal $\Bary$, hence we may apply~\cite[Proposition 5.2]{CM10} to conclude that the ``free boundary'' (as defined in~\cite{CM10}) in $\spt{\Dominating[j]}$ enjoys the same semiconcavity. However, since one cannot make any assumptions about convexity of $\spt{\Bary}$ and bounds on the density of $\Bary$, arguments based on Caffarelli's regularity theory (see~\cite{Caf92}) to obtain higher regularity of the free boundary cannot be applied.
\end{rmk}
\section{Extension to more general cost functions}\label{sec: general costs}
Here we mention that our main result can be extended to a more general class of cost functions.  Consider a cost function of the form 

\begin{equation}\label{hp}
c(x_1,\ldots, x_\Number)= \inf_{y\in Y}\sum_{j=1}^Nc_j(x_j,y),
\end{equation}
where $c_j:\Omega_j\times Y\to \R$ for some fixed, open domains $\Omega_j$ and $Y$. 
Also consider the generalized partial barycenter problem:

\begin{equation}\label{generalbc}
\min_{\Bary\in\Constraint}\sum_{j=1}^\Number\mathcal{T}_{c_j, m}(\Dominating[j],\Bary)
\end{equation}

where $\mathcal{T}_{c_j, m}$ is the partial transport distance with cost $c_j$:
\begin{equation*}
\mathcal{T}_{c_j, m}(\Dominating[j],\Bary):=\min_{\TwoCoupling'\in\PartialCouples{\Dominating[j]}{\Bary},\  \mass{\TwoCoupling'}=m}\int_{\R^n\times\R^n} c_j(x_j,y)\TwoCoupling'(dx, dy).
\end{equation*}

In order to obtain the equivalent of Proposition~\ref{prop: BCm=MMm}, we will require the following assumptions:

\begin{description}
\item[{[H1]}] For all $j$, the costs $c_j$ are $C^2$ and $\det{D^2_{x_jy}c_j}\neq 0$ on $\Omega_j\times Y$.
\item[{[H2]}] For all $j$, the mappings $D_{x_j}c_j(x_j, \cdot)$ are injective for each $x_j$, and $D_{y}c_j(\cdot, y)$ are injective for each $y$
\item[{[H3]}] For each $(x_1,x_2,\ldots, x_\Number)$ the infimum in \eqref{hp} is attained by a \textit{unique} $y=\overline{y}(x_1,x_2,\ldots, x_\Number) \in Y$.
\item[{[H4]}] The matrix $B(x_1,x_2,\ldots, x_\Number):=\sum_{i=1}^\Number D^2_{yy}c_j(x_j,\overline{y}(x_1,x_2,\ldots ,x_\Number))$ is non-singular.
\end{description}
Additionally, to obtain the equivalent of Theorem~\ref{thm: partial barycenter is unique}, we must also assume the following condition:
\begin{description}
\item[{[H5]}] For all $j$, we have $c_j(x_j, y)\geq 0$ with $c_j(x_j, y)=0$ if and only if $x_j=y$.
\end{description}
Under these assumptions, we can generalize the main results to a more general class of cost functions.
\begin{thm}\label{theorem: general costs}
 Fix compactly supported, absolutely continuous measures $\Dominating[j]$, for $j=1,\ldots,\Number$, let  $0\leq m\leq \min_{1\leq j\leq \Number}{\{\mass{\Dominating[j]}\}}$, and assume that the cost functions $c_j$ for $1\leq j\leq \Number$ satisfy conditions {\bf[H1]-[H5]} above. Then the following hold:
 \begin{enumerate}
  \item If $\MMCoupling$ is a solution to~\eqref{eqn: MMm} with cost \eqref{hp}, then $\nu =\bar y_{\#}\MMCoupling$ is a solution to \eqref{generalbc}.  On the other hand, if $\nu$ is a solution to  \eqref{generalbc}, then $\MMCoupling=\Paren{\MultiDualMapping[\nu]{1},\ldots,\MultiDualMapping[\nu]{\Number}}_{\#}\nu$ is a solution to ~\eqref{eqn: MMm}, where $\RevMultiMappingCoupling[\nu]{j}_\#\nu$ is the (unique) minimizer in $\mathcal{T}_{c_j, m}\Paren{\Dominating[j], \Bary}$.
    \item Assume in addition that $\m\geq \mass{\DominatingCommon}$ where $\DominatingCommon$ is the measure with density $\min_{1\leq j\leq \Number}{\DominatingDens[j]}$. Then both the multi-marginal optimal partial transport problem~\eqref{eqn: MMm} with cost function given by \eqref{hp}, and the generalized partial barycenter problem \eqref{generalbc} admit unique solutions.  
  \end{enumerate}
\end{thm}
The proof of the preceding theorem is a straightforward adaptation of the proof of our main results.  The necessary ingredients are a relationship between the partial transport problem and the partial barycenter problem, established by Carlier and Ekeland \cite[Proposition 3]{CE10}.  The conditions {\bf [H1] - [H4]} guarantee the uniqueness of the solution to the standard multi-marginal problem with cost \eqref{hp}, and absolute continuity of the standard generalized barycenter \cite{P13a}, necessary ingredients in our argument here.

On the other hand, condition {\bf [H5]} (together with {\bf [H1]} and {\bf [H2]}) are necessary to invoke the results of Figalli in the two marginal case~\eqref{eqn: OTm} that we have relied on to prove the results of Section~\ref{sec: uniqueness in BCm} (see~\cite[Remark 2.11]{Fig10a} for details).
\begin{appendices}

\section{Summary of notation}\label{sect: notation}
We provide here a brief summary of the main notation used in the proof of Theorem \ref{thm: partial barycenter is unique}.  We generally associate $i\in \{0, 1\}$, and superscripts to objects related to the barycenter measures, while we associate $j\in \{1,\ldots, N\}$ and subscripts to objects related to the marginal measures.
\begin{itemize}
\item $\Dominating[j]$, for $j=1,...,N$, is a marginal in the original problem.  Its density is denoted by $\DominatingDens[j]$.
\item $\Bary[i]$, for $i=0,1$, is a measure of total mass $m$.  It is an argument (usually a minimizer) of the functional $F$.  Its density is denoted by $\BaryDens[i]$.
\item $\Bary[t]:=t\Bary[1] +(1-t)\Bary[0]$, for $t \in [0,1]$, is a convex interpolant of $\Bary[0]$ and $\Bary[1]$.  Its density is denoted by $\BaryDens[t]$.
\item $\PairedMulti[i]{j}$, for $i=0,1$ and $j=1,....,N$, are the active submeasures of the $\Dominating[j]$, paired with $\Bary[i]$ via the (two marginal) optimal partial transport plan $\RevMultiMappingCoupling[i]{j}_\#\Bary[i]$  (thus $\PairedMulti[i]{j}=\Paren{\MultiDualMapping[i]{j}}_\#\Bary[i]\leq \Dominating[j]$).   Its density is denoted by $\PairedMultiDens[i]{j}$.  Note that the $\leq$ symbol is meant to indicate that $\PairedMulti[i]{j}$ is a submeasure of $\Dominating[j]$.
\item $\DominatingCommon$ is the common mass of the marginals $\Dominating[j]$, for $j=1,...N$; that is, the measure with density $\DominatingCommonDens = \min_{j=1,...N}\DominatingDens[j]$.
\item $\BaryCommon$ is the common mass of  $\Bary[0]$ and $\Bary[1]$; that is, the measure with density $\BaryCommonDens = \min_{i=0,1}\BaryDens[i]$.
\item $\MultiDualMapping[i]{j}$ is the optimal partial transport mapping from $\Bary[i]$ to $\Dominating[j]$; this mapping pushes $\Bary[i]$  forward to $\PairedMulti[i]{j}$.
\item $ \StaySet:=\{x\in\R^n\mid \BaryDens[0](x)\neq\BaryDens[1](x)\}$ is the set where the densities of $\Bary[0]$ and $\Bary[1]$ differ.
\item $\MinSmaller$ and $\MinBigger$ (for one of $i=0$ or $1$) are certain sets where the common density $\DominatingCommonDens$ is, respectively, smaller and larger that the density $\BaryDens[i]$ (note that these sets are chosen to also have additional special properties, see Claim 1 in the proof of Theorem~\ref{thm: partial barycenter is unique}).

\end{itemize}
\section{Proof of claims used in the proof of Theorem~\ref{thm: partial barycenter is unique}}\label{sect: proof of claims}
Here we prove the claims used in the proof of the main theorem.
\begin{proof}[{Proof of Claim 1}]
Note that as we choose  $\MinBigger\subset\StaySet$, we can write $\R^n\setminus\MinBigger$ as the disjoint union of the sets $\Curly{\BaryDens[i]=0}\setminus\MinBigger$, $\Curly{\BaryDens[i]>0}\cap\Paren{\StaySet\setminus\MinBigger}$, and $\Curly{\BaryDens[i]>0}\setminus\StaySet$. Using again the fact that $\MinBigger\subset\StaySet$, by~\eqref{eqn: identity map} we have
\begin{align*}
&\Leb{\MultiDualMapping[i]{j}\Paren{\Curly{\BaryDens[i]=0}\setminus\MinBigger}\cap\MultiDualMapping[i]{j}(\Paren{\MinBigger}}\\
&=\Leb{\Paren{\Curly{\BaryDens[i]=0}\setminus\MinBigger}\cap\MinBigger}=0
\end{align*}
and likewise
\begin{align*}
\Leb{\MultiDualMapping[i]{j}\Paren{\Curly{\BaryDens[i]>0}\cap\Paren{\StaySet\setminus\MinBigger}}\cap \MultiDualMapping[i]{j}\Paren{\MinBigger}}=0.
\end{align*}
Thus to guarantee~\eqref{eqn: doesn't get mass from elsewhere} it would be sufficient to show we can choose $\MinBigger$ such that
\begin{align}\label{eqn: show this instead}
  \Leb{\MultiDualMapping[i]{j}(\Curly{\BaryDens[i]>0}\setminus\StaySet)\cap\MultiDualMapping[i]{j}(\MinBigger)}=0.
\end{align}
Now, by~\eqref{eqn: barycenter dominated by pointwise minimum on stay set} and since $\Bary[0]\neq\Bary[1]$, there must exist a positive measure subset of $\StaySet$ on which $\DominatingCommonDens>\BaryDens[i]$ for either of $i=0$ or $1$. By the definition of $\StaySet$, we can take $\MinBigger$ to be a subset of this aforementioned set, in such a way that $\BaryDens[i']>0$ on $\MinBigger$, for at least one of $i'=0$ or $1$ ($i'$ may or may not be equal to $i$).  Indeed, for any $x\in \StaySet$ either $\BaryDens[0](x)>0$ or $\BaryDens[1](x)>\BaryDens[0](x)=0$. Since $\Leb{\StaySet}>0$, we can choose $i'$ appropriately.

At this point, by shrinking $\MinSmaller$ as necessary, we can ensure~\eqref{eqn: bigger smaller sets} holds. Now, since $\{\BaryDens[i]>0\}\setminus\StaySet=\{\BaryDens[i']>0\}\setminus\StaySet$ by the  definition of $\StaySet$, and using the fact that $\MultiDualMapping[i]{j}=\MultiDualMapping[i']{j}$ a.e.,  we see
\begin{align*}
 &\Leb{\MultiDualMapping[i]{j}(\{\BaryDens[i]>0\}\setminus\StaySet)\cap\MultiDualMapping[i]{j}(\MinBigger)}\\
 &=\Leb{\MultiDualMapping[i']{j}(\{\BaryDens[i']>0\}\setminus\StaySet)\cap\MultiDualMapping[i']{j}(\MinBigger)}.
\end{align*}
Finally, as $\MultiDualMapping[i']{j}$ is injective a.e. on $\{\BaryDens[i']>0\}$, and both $\{\BaryDens[i']>0\}\setminus\StaySet$ and  $\MinBigger$  are  disjoint subsets of $\{\BaryDens[i']>0\}$
we obtain  
\begin{align*}
\Leb{\MultiDualMapping[i']{j}(\{\BaryDens[i']>0\}\setminus\StaySet)\cap\MultiDualMapping[i']{j}(\MinBigger)}=0,
\end{align*}
 which implies ~\eqref{eqn: show this instead}.
\end{proof}
\begin{proof}[{Proof of Claim 2}]
Fix an arbitrary measurable set $E$, then for any $1\leq j\leq N$
\begin{align*}
 \Paren{\MultiDualMapping[i]{j}}_\#{\BaryBar}(E)&=\BaryBar(\MultiDualMappingInv[i]{j}(E))  \\
&=\Bary[i](\MultiDualMappingInv[i]{j}(E))+\int_{\MinBigger\cap\MultiDualMappingInv[i]{j}(E)}(\DominatingCommonDens-\BaryDens[i])dx-\int_{\MinSmaller\cap\MultiDualMappingInv[i]{j}(E)}(\BaryDens[i]-\DominatingCommonDens)dx\notag\\
 &\leq \PairedMulti[i]{j}(E)+\Dominating[j](\MinBigger\cap \MultiDualMappingInv[i]{j}(E))-\Bary[i](\MinBigger\cap \MultiDualMappingInv[i]{j}(E))
 \end{align*}
where, in the last line, we have used that  $(\MultiDualMapping[i]{j})_\#{\Bary} = \PairedMulti[i]{j}$, that $\DominatingCommonDens \leq \DominatingDens[j]$ for all $j$, and that the last term in the second line is nonpositive.

 Now since $\MinBigger\subset\StaySet$, by~\eqref{eqn: identity map} we see that (up to null sets) we have
\begin{align}\label{eqn: inclusion up to null}
 \MinBigger\cap E= \MinBigger\cap \MultiDualMappingInv[i]{j}(E) 
\end{align}
In addition we find (again up to null sets),
\begin{align*}
 \MultiDualMappingInv[i]{j}(\MinBigger\cap E)&=\Paren{\MultiDualMappingInv[i]{j}(\MinBigger\cap E)\cap \MinBigger}\cup \Paren{\MultiDualMappingInv[i]{j}(\MinBigger\cap E)\setminus \MinBigger}\\
 &\subset\Paren{\MinBigger\cap E}\cup \Paren{\MultiDualMappingInv[i]{j}(\MinBigger)\setminus \MinBigger}.
\end{align*}
 Then by \eqref{eqn: doesn't get mass from elsewhere} and the absolute continuity of $\PairedMulti[i]{j}$ we find that
\begin{align*}
 0&=\PairedMulti[i]{j}\Paren{\MultiDualMapping[i]{j}(\R^n\setminus\MinBigger)\cap\MinBigger}\\
 &=\Bary[i]\Paren{\MultiDualMappingInv[i]{j}\Paren{\MultiDualMapping[i]{j}(\R^n\setminus\MinBigger)\cap\MinBigger}}\\
 &\geq \Bary[i]\Paren{(\R^n\setminus\MinBigger)\cap\MultiDualMappingInv[i]{j}(\MinBigger)},
\end{align*}
hence
\begin{align*}
 \Bary[i]\Paren{\MultiDualMappingInv[i]{j}(\MinBigger\cap E)}\leq \Bary[i]\Paren{\MinBigger\cap E}.
\end{align*}


Then by combining with \eqref{eqn: inclusion up to null} we can calculate
\begin{align*}
 &\PairedMulti[i]{j}(E)+\Dominating[j](\MinBigger\cap \MultiDualMappingInv[i]{j}(E))-\Bary[i](\MinBigger\cap \MultiDualMappingInv[i]{j}(E))\\
 &\leq \PairedMulti[i]{j}(E)+\Dominating[j](\MinBigger\cap E)-\Bary[i](\MultiDualMappingInv[i]{j}(\MinBigger\cap E))\\
 &=\PairedMulti[i]{j}(E)+\Dominating[j](\MinBigger\cap E)-\PairedMulti[i]{j}(\MinBigger\cap E)\\
 &=\PairedMulti[i]{j}((  \mathbb{R}^n \setminus \MinBigger)\cap E)+\Dominating[j](\MinBigger\cap E)\\
 &\leq \Dominating[j](E),
\end{align*}
 proving the claim.
\end{proof}

\begin{proof}[{Proof of Claim 3}]
First we calculate for each $1\leq j\leq \Number$,
\begin{align}
 \Cost[2]{\RevMultiMappingCoupling[i]{j}_\#\BaryBar}&=\Cost[2]{\RevMultiMappingCoupling[i]{j}_\#\Bary[i]}+\int_{\MinBigger}\norm{\MultiDualMapping[i]{j}(x)-x}^2(\DominatingCommonDens-\BaryDens[i])dx\notag\\
 &\qquad-\int_{\MinSmaller}\norm{\MultiDualMapping[i]{j}(x)-x}^2(\BaryDens[i]-\DominatingCommonDens)dx\notag\\
 &=\PartialWassSq{\Dominating[j]}{\Bary[i]}-\int_{\MinSmaller}\norm{\MultiDualMapping[i]{j}(x)-x}^2(\BaryDens[i]-\DominatingCommonDens)dx,\label{eqn: modified total cost}
\end{align}
where we have again used~\eqref{eqn: identity map} and that $\MinBigger\subset\StaySet$. Now by definition of $\MinSmaller$, we must have $-\int_{\MinSmaller}\norm{\MultiDualMapping[i]{j}(x)-x}^2(\BaryDens[i]-\DominatingCommonDens)dx\leq 0$ for each $j$, and so  ~\eqref{eqn: modified total cost} implies
\begin{equation}\label{eqn: modified total cost 2}
 \Cost[2]{\RevMultiMappingCoupling[i]{j}_\#\BaryBar} \leq\PartialWassSq{\Dominating[j]}{\Bary[i]}.
\end{equation}
 Suppose that there is equality for every $1\leq j\leq \Number$; that is, that 
$$
-\int_{\MinSmaller}\norm{\MultiDualMapping[i]{j}(x)-x}^2(\BaryDens[i]-\DominatingCommonDens)dx= 0
$$
for each $j$.  This would imply that $\MultiDualMapping[i]{j}$ is the identity map a.e. on $\MinSmaller$ for every $j$ as well. However, there exists some set $\mathcal{A}_{j'}\subset\MinSmaller$ with strictly positive measure on which $\DominatingCommonDens\equiv\DominatingDens[j']$ for some index $1\leq j'\leq \Number$. This would imply that (using the a.e. injectivity of $\MultiDualMapping[i]{j'}$ on $\spt{\Bary[i]}$)
\begin{align*}
 \PairedMulti[i]{j'}(\mathcal{A}_{j'})&=\Bary[i](\MultiDualMappingInv[i]{j'}(\mathcal{A}_{j'}))\\
 &=\int_{\mathcal{A}_{j'}}\BaryDens[i]dx\\
 &>\int_{\mathcal{A}_{j'}}\DominatingCommonDens dx\\
 &=\Dominating[j'](\mathcal{A}_{j'}),
\end{align*} 
contradicting that $\PairedMulti[i]{j'}\leq \Dominating[j']$. Thus, we must have \emph{strict} inequality in ~\eqref{eqn: modified total cost 2} for at least one $j$, and so by  summing~\eqref{eqn: modified total cost 2} over $1\leq j\leq\Number$, we 
obtain~\eqref{eqn: new map too cheap}. 
\end{proof}

\end{appendices}
\bibliography{partialmultibiblio}
\bibliographystyle{plain}
\end{document}